\documentclass[11pt]{amsart}
\usepackage{amssymb, amsmath, amsthm}
\usepackage[latin1]{inputenc}
\usepackage{graphicx}
\usepackage[final]{hyperref}
\usepackage{color}
\usepackage{ulem}
\usepackage{epstopdf}
\usepackage{cancel}

\usepackage[a4paper, centering]{geometry}
\usepackage{color}
\usepackage{graphicx}
\usepackage{scalerel,stackengine}

\usepackage{mathtools}

\newtheorem{theorem}{Theorem}
\newtheorem{proposition}[theorem]{Proposition}
\newtheorem{examples}[theorem]{Examples}
\newtheorem{lemma}[theorem]{Lemma}
\newtheorem{corollary}[theorem]{Corollary}

\theoremstyle{definition}

\theoremstyle{remark}
\newtheorem{remark}[theorem]{Remark}

\definecolor{verde}{RGB}{20,150,100}
\definecolor{purple}{RGB}{200,30,200}

\newcommand{\EEE}{\color{black}}

\stackMath
\newcommand\reallywidecheck[1]{%
\savestack{\tmpbox}{\stretchto{%
  \scaleto{%
    \scalerel*[\widthof{\ensuremath{#1}}]{\kern-.6pt\bigwedge\kern-.6pt}%
    {\rule[-\textheight/2]{1ex}{\textheight}}
  }{\textheight}%
}{0.5ex}}%
\stackon[1pt]{#1}{\scalebox{-1}{\tmpbox}}%
}

\def\R{\mathbb{R}}

\def\N{\mathbb{N}}

\newcommand{\sm}{\setminus}
\newcommand{\vps}{\varepsilon}
\newcommand{\Om}{\Omega}
\newcommand{\sq}{\subseteq}
\newcommand{\ra}{\rightarrow}
\newcommand{\reflext}{{\mathcal R} _t} 
\newcommand{\reflexT}{{\mathcal R} _T}

\def \e{\varepsilon}

\begin{document}
\title[]{ 
 Rigidity for measurable sets 
 }

\author[]{Dorin Bucur, Ilaria Fragal\`a}

\thanks{}

\address[Dorin Bucur]{
Universit\'e  Savoie Mont Blanc, Laboratoire de Math\'ematiques CNRS UMR 5127 \\
  Campus Scientifique \\
73376 Le-Bourget-Du-Lac (France)
}
\email{dorin.bucur@univ-savoie.fr}

\address[Ilaria Fragal\`a]{
Dipartimento di Matematica \\ Politecnico  di Milano \\
Piazza Leonardo da Vinci, 32 \\
20133 Milano (Italy)
}
\email{ilaria.fragala@polimi.it}

\date{\today}


\maketitle

\begin{abstract}  
Let $\Omega \subset \R^d$, $d\geq 2$, be a set with finite Lebesgue measure such that,  for a fixed radius $r>0$, the Lebesgue measure of $\Omega \cap B _ r (x)$ is equal to a  positive   constant when $x$ varies in the essential boundary of $\Om$.  We prove that $\Omega$ is a ball (or a finite union of  equal  balls) provided it satisfies a nondegeneracy condition, which holds in particular for any set of diameter larger than $r$ which is either  open and connected, or of finite perimeter and indecomposable.  The proof   requires reinventing each step  of the moving planes method by Alexandrov in the framework of measurable sets. 
\end{abstract}



\section{Introduction}\label{sec:main}

In this paper we study the following rigidity problem for measurable sets of the Euclidean space $\R ^d$: 
for a given radius $r>0$,  identify measurable sets $\Omega$ of  finite Lebesgue measure such that,
 for a positive constant $c$,  
\begin{equation}\label{f:hyp3} 
|\Omega \cap B _ r (x) | = c \qquad \forall x \in \partial^* \Omega\,,
\end{equation}
where $B _ r (x)$ is the ball of radius $r$ centred at $x$,  $| \cdot |$ denotes the Lebesgue measure, and  $\partial ^* \Omega$ indicates the essential boundary of $\Omega$ ({\it i.e.}\ the set of  points $x\in \R^d$  at which  both $\Omega$ and its complement $\Omega ^ c$ have a strictly positive $d$-dimensional upper density).

Such an easy-to-state geometric question conceals several relations with classical questions in Differential Geometry, 
as well as with recent advances in Convex Geometry and Geometric Measure Theory. 
We outline them below, before stating the results.

\bigskip 
{\bf A quick historical overview.} The question about rigidity criteria  obtainable
by measuring intersections of a domain $\Omega$ in $\R^ d$ with balls rolling along its boundary dates back to almost one century ago. 
The idea is to look at the behaviour, for $x \in \partial \Omega$,  of 
 surface integrals 
 ${ \mathcal H ^ { d-1}(  \Omega \cap \partial B _ r (x)) }$, or of volume integrals 
 $|\Omega \cap B _ r (x)|$. 
In its first grain,  
this idea can be found in a paper dating back to 1932 by Cimmino (see \cite{Ci32}),  
where he raised the following question:   is it possible to characterize 
 surfaces $\Gamma$ which bisect the $\mathcal H ^ 2$-measure 
 of the boundary of any ball which is centred on $\Gamma$ and has a sufficiently small radius?     
Cimmino's problem was solved more than sixty years later by Nitsche, who proved that the only (smooth)  surfaces with this property are the plane and the helicoid \cite{nitsche}. 

The problem reemerged in the first  2000s under different garments:  motivated by the study of 
isothermic surfaces in heat conduction, also in relation with the so-called {\it Matzoh ball soup} problem  \cite{MS},
 Magnanini, Prajapat and Sakaguchi were led to consider {\it $B$-dense domains}, namely subsets $\Omega$ in $\R^d$ such that,
 for any  $r>0$, there exists a positive constant $c(r)$ such that 
   $|\Omega \cap B _ r (x)| = c (r)$ for every $x \in \partial \Omega$.  In particular they proved that, if the boundary of a $B$-dense domain $\Omega$ is a complete embedded minimal surface of finite total curvature in $\R^3$, $\Omega$ must be a plane \cite{MPS06}.  
  Later in 2016,  Magnanini and Marini proved that, in any space dimension and for any given convex body $K$, if a set $\Omega $ of positive and finite Lebesgue measure is $K$-dense (meaning that $|\Omega \cap (x+ rK)| = c (r)$ for every $x \in \partial \Omega$), then $\Omega $ and $K$ are homothetic ellipsoids  \cite{MM16} (see also \cite{MM14, ABG}).  Note that, as long as it is assumed to hold for any sufficiently small $r>0$, the constancy of the volume measure 
$|\Omega \cap B _ r (x)|$ is actually equivalent to that of the surface measure  ${ \mathcal H ^ { d-1}(  \Omega \cap \partial B _ r (x)) }$; so the difference between $B$-dense domains and those considered by Cimmino is just that the volume fraction is no longer fixed to $\frac{1}{2}$ (in other words, the constant $c(r)$  may differ from $\frac{1}{2} |B _ r (x)|$).


All the rigidity results mentioned so far
are naturally related to 
a central question in Differential Geometry, namely the classification of hypersurfaces with constant mean curvature in $\R ^ d$.  Indeed, 
in view of the asymptotic expansion  
\begin{equation}\label{f:ae} | \Omega\cap B _ r (x)|= \frac{1}{2}\omega_d  r^ d - \frac {d-1}{2(d+1)}\omega_{d-1} H_\Omega (x) r^{d+1}+ O (r^{d+2}),
\end{equation} 
where $\omega_d$  is the volume of the unit ball in $\R^d$ and $H _\Omega$ is the mean curvature of $\partial \Omega$  \cite{HT03},  
the Lebesgue measure of $\Omega \cap B _ r (x)$ can be interpreted as an integral approximation of $H _\Omega(x)$;
differentiating with respect to $r$, the same assertion is valid for the $(d-1)$-dimensional measure of $\Omega \cap \partial B _ r (x)$. 

In this perspective, the rigidity criteria quoted above can be read as counterparts of some cornerstone results for hypersurfaces with constant mean curvature. 
Thus, Magnanini-Marini criterion reminds the celebrated theorem proved in 1958  by Alexandrov \cite{ale}:
if $\Omega$ is a bounded connected domain of class $\mathcal C ^ 2$ such that $\partial \Omega$ has  constant mean curvature, then $\Omega$ is a ball (a generalization has been proved very recently by Delgadino-Maggi \cite{delmag}, showing that any set with finite Lebesgue measure and finite perimeter with constant distributional mean curvature is a finite union of equal balls).   
Likewise, Nitsche criterion, though not involving any topological constraint, 
reminds the much harder problem, settled only in 2005 by Meeks-Rosenberg  \cite{meeks}, 
of classifying the plane and the helicoid as the unique simply connected  minimal surfaces embedded in $\R ^ 3$.  

Comparing the constancy of the mean curvature
with the constancy of one of the measures ${ \mathcal H ^ { d-1}(  \Omega \cap \partial B _ r (x)) }$  and $|\Omega \cap B _ r (x)|$  for any sufficiently small $r>0$, 
it is clear that the former is in principle weaker, as it concerns only {\it one}  among the coefficients of the expansion \eqref{f:ae} in powers of $r$; 
on the other hand, defining a notion of mean curvature requires some boundary regularity, even if done in distributional sense, 
while measuring intersections with balls requires no smoothness at all. 

All in all, at present no rigidity result seems to be available for arbitrary measurable sets under a fairly weak condition such as the constancy of a single and well-defined quantity. 
Aim of this paper is to provide a first contribution in this direction, by considering sets which satisfy condition \eqref{f:hyp3}.  The fact that we work with {\it one fixed radius} makes the approach completely new. We call sets satisfying \eqref{f:hyp3} {\it $r$-critical}.  
Note that, since the essential boundary $\partial ^*\Omega$ is included in the topological boundary $\partial \Omega$, condition \eqref{f:hyp3} is weaker than the  constancy of $|\Omega \cap B _ r (x)|$ along $\partial \Omega$. 
Incidentally let us also mention that, at least for convex domains, $r$-criticality  can be rephrased by saying that $\partial \Omega$ is a level surface
for the {\it cross-covariogram function} of $\Omega$ and $B _ r (0)$. (The cross-covariogram function of two convex bodies  $K_1$ and $K_2$ is  defined as  
 $g_{(K_1, K_2)} (x) := |K_1 \cap ( x +K_2)|$,  and the investigation of its level lines has attracted some attention in the literature on Convex Geometry, also in connection with  the floating body problem, see for instance \cite{bianchi, MRS93}). 

The reason for the terminology ``$r$-critical''' is that, notably,  this notion has still a variational interpretation.   
Actually, sets of constant mean curvature may be viewed as stationary sets for the perimeter functional under a volume preserving perturbation. From this point of view, Alexandrov result, along with its extension in \cite{delmag}, allows to identify critical sets for the isoperimetric inequality proved in 1958 by De Giorgi \cite{degiorgi, fusco}. 

An interpretation in the same vein can be given to $r$-critical sets, as soon as the isoperimetric inequality is replaced by another classical one, which is even more ancient, 
namely the {\it rearrangement inequality} proved in 1932 by Riesz  \cite{riesz}. In a simplified version  it states that, for any 
 radially symmetric, decreasing, non negative function $h$, 
balls maximize, under a constraint of prescribed Lebesgue measure,   the integral functional 
$$
 J_ h  (\Omega) := \int _\Omega \int _{\Omega} h(x-  y)  \, dx \, dy.
$$
Given an integrand $h$ as above,  it is not difficult to check that
balls maximize $J _h$ provided  they maximize  $J_ {\chi _{B _ r (0)}} $  for all $r>0$. 
Hence  the choice $h = \chi _{B _ r (0)}$ is of special relevance, and for such kernel
 stationary domains  are precisely sets satisfying condition \eqref{f:hyp3}. 
 
 Equivalently, in view of the equality  $|\Omega\cap B _ r (x)|-|\Omega^c\cap B _ r (x)| = 2 |\Omega\cap B _ r (x)|-\omega_d r^d  $, 
$r$-critical sets may be viewed as stationary domains,  under volume preserving perturbations, for the  {\it nonlocal perimeter}
$$r{\text -{\rm Per}} (\Omega):= \int _\Omega \int _{\Omega^c} \chi  _{ \{|x-y|< r\} }  \, dx \, dy\,;$$ 
in fact, the quantity $|\Omega^c \cap B _ r (x)|-|\Omega\cap B _ r (x)|$ fits 
 the definition  of {\it generalized nonlocal mean curvature} proposed by 
Chambolle, Morini and Ponsiglione in \cite{CMP15} (see also \cite{MRT}). 

This nonlocal interpretation brings immediately to mind the fractional perimeter introduced in the seminal papers \cite{CS08, CRS10},
$$P_s(\Omega)= \int _\Omega \int _{\Omega^c}\frac{1}{ |x-y|^{d+2s} } \, dx \, dy\, , \qquad s \in \big (0, \frac 12 \big )\,,$$
and particularly two independent results appeared in  2018 by  Ciraolo-Figalli-Maggi-Novaga  \cite{CFMN18} and by Cabr\'e-Fall-Morales-Weth   \cite{CMMW}, which have identified balls as the unique stationary domains of class $\mathcal C^{1, \alpha}$ ($\alpha > 2s$) for the perimeter $P _ s$. 
 
The qualitative properties of the kernel $\chi _ { B _ r(0)}$  make our problem dramatically different:  it is bounded
(allowing to deal with arbitrary measurable sets), compactly supported (producing short-range nonlocal effects),   and
 discontinuous with level sets of positive measure  (enhancing the need for some ``transmission'' issue, companion to $r$-criticality, in order to get rigidity).  
Fatally, notwithstanding the differences between the two questions, the investigation of $r$-critical measurable sets evokes Pompeiu problem. 

Stationary domains for more general kernels $h$ could be considered, see Remark \ref{rem:h}, but for the sake of clearness we prefer to 
 focus on the choice $h =\chi _ { B _ r(0)}$,  since it allows to capture all new relevant ideas.  
  
To conclude, our way  to rigidity appears to be very suitable to deal with from an applied point of view. 
As a matter of fact, in recent times the Lebesgue measure $|\Omega \cap B _ r (x)|$ has been successfully exploited
 in  Geometry Processing, under the name of {\it volumetric integral invariant},
 for an efficient computation avoiding noise of shape characteristics  (including the mean curvature), with applications to feature extraction at multiple scales and automatic rassembling of fragments of broken objects (see \cite{MCHYS, PWHY}).

\bigskip
{\bf The results.} Before stating our main result, in order to introduce  the key condition for rigidity, companion to  $r$-criticality,  we set the following definition:  
%
%
we say that a measurable set $\Omega$ in $\R^d$ 
is {\it $r$-degenerate} if
$$\inf _{x_1, x_2 \in \partial ^* \Omega} \frac{\big |\Omega \cap ( B _ r (x_1) \Delta B _ r (x_2)) \big |  }{\| x_1 -x_2 \| }  = 0\,. 
$$
  A discussion about this  notion  is postponed to Section \ref{sec:fat}.  Therein we shall provide, in particular, a measure theoretic condition sufficient for nondegeneracy,  which permits to show that bounded open connected sets, as well as bounded indecomposable sets with finite perimeter, are not degenerate for any $r$ smaller than their diameter.


\begin{theorem} \label{t:serrin3}
 Let $\Omega$ be a measurable set with finite Lebesgue measure in $\R^d$, and let $r>0$.  Assume that $\Omega$ is $r$-critical and  not $r$-degenerate.   
Then $\Omega$ is equivalent to  a finite union of balls of the same radius $R> \frac{r}{2}$, at mutual distance  larger than or equal to $r$. 
\end{theorem}

\begin{remark} \label{r:necessary}
Rigidity may fail if the finite measure assumption is dropped:
any  halfspace or any   strip 
 $\{ x \in \R^d : a<x_1<b\}$ is   critical and not degenerate for any $r >0$. 
  As well, rigidity may fail  for $r$-critical sets of finite measure which are   $r$-degenerate; 
  for some examples in this respect, see Section \ref{sec:fat}.
 \end{remark}

 \begin{remark} A result analogue to Theorem \ref{t:serrin3} can be immediately deduced, by using the area formula, if balls are replaced by ellipsoids: 
if $E$ is a given ellipsoid,  any set with finite measure which satisfies the criticality and  degeneracy   conditions with $x+ E$ in place of $B _ r (x)$,  is a finite union of ellipsoids homotetic to $E$.  
\end{remark}

\begin{remark}\label{r:fractional} We point out that the initial choice of the radius $r$ produces a sort of  {\it tuned bubbling phenomenon}, which may occur only with 
a precise  lower threshold both on the size of the balls and on their mutual distance. In accordance with the short-range nonlocal nature of our kernel, this 
behaviour should be compared with the local result in \cite{delmag}, where bubbling can occur at any scale, and the fractional results in \cite{CFMN18, CMMW}, where bubbling cannot occur at all. Let us also mention that a further motivation for characterising finite unions of equal balls is their appearance  as optimal domains  in spectral shape optimization problems (see \cite{H06, BH19}). 
  \end{remark}

\smallskip 
We now present 
some consequences of Theorem \ref{t:serrin3} for sets enjoing some kind of regularity.  
We begin by the case of open sets: 

\begin{corollary}\label{t:serrin1}
Let $\Omega$ be an open set  with finite Lebesgue measure in $\R^d$, and let $r>0$. 
Assume   that there exists a positive constant $c$ such that
\begin{equation}\label{f:hyp1} 
|\Omega \cap B _ r (x) | = c \qquad \forall x \in \partial \Omega\,, 
\end{equation}
where $\partial \Omega$ denotes the topological boundary. 

If $r < \inf_i \{ {\rm diam} (\Omega _i ) \}$, where $\Omega _i$ are the open connected components of $\Omega$, then  
$\Omega$ is a finite union of balls of the same radius $R> \frac{r}{2}$, at mutual distance  larger than or equal to $r$.  
In particular, if $\Omega$ is connected and $r < {\rm diam} (\Omega)$, then $\Omega$ is a ball. 
\end{corollary}

Next we turn to the case of sets with finite perimeter. Recall that, 
 following \cite{ACMM01},  any set $\Omega$ with finite perimeter can be written as finite or countable family of  {\it indecomposable components} $\Omega _i$. This means that  each $\Omega_i$ is indecomposable in the sense that it does not admit a partition $(\Omega _i ^+, \Omega _i ^-)$, with  $|\Omega _i ^ \pm| >0$ and 
${\rm Per } (\Omega_i)= {\rm Per } (\Omega_i^+) + {\rm Per } (\Omega _ i ^-)$, and that the $\Omega _i$'s are maximal indecomposable sets. 
Recall also that  the {\it reduced boundary} $\mathcal F \Omega$   is  the collection of points $x \in {\rm supp} ( D \chi _ \Omega)$ such that the generalized  normal 
$\nu _\Omega (x):= \lim  _{\rho \to 0} { D \chi _{\Omega}   (B _\rho (x)) } / { |D \chi _{\Omega} |  (B _\rho (x) )}$
 exists in $\R^d$ and satisfies $\nu _\Omega (x) = 1$. 
For sets of finite perimeter, a sufficient condition for $r$-criticality is the validity of condition 
\eqref{f:hyp2} below, because the closure of
$\mathcal F \Omega \setminus N$, for any  $\mathcal H ^ { d-1}$-negligible set $N$, turns out to contain $\partial ^ * \Omega$. 
 
 \begin{corollary}\label{t:serrin2}
 Let $\Omega$ be a  set of finite perimeter and  finite Lebesgue measure in $\R^d$, and let $r>0$.  Assume there exists a positive constant $c$ such that
\begin{equation}\label{f:hyp2} 
|\Omega \cap B _ r (x) | = c \qquad  \text{for $\mathcal H ^ { d-1}$-a.e.}\  x \in \mathcal F \Omega\,.
\end{equation}

If  $r < \inf_i \{ {\rm diam} (\Omega _i ) \}$, where $\Omega _i$ are the indecomposable components of $\Omega$, then  
$\Omega$ is equivalent to  a finite union of balls of the same radius $R> \frac{r}{2}$, at mutual distance  larger than or equal to $r$.  
In particular, if $\Omega$ is indecomposable and $r < {\rm diam} (\Omega)$, $\Omega$ is a ball. 
\end{corollary}

 \bigskip  
{\bf About the proof of Theorem \ref{t:serrin3}.}  Alexandrov rigidity theorem was obtained by a very elegant proof, based on what he called {\it reflection principle}, nowadays commonly known as {\it the moving planes method}. His  brilliant idea was destined to have a tremendous impact also in the field of Mathematical Analysis: its
implications in PDEs were firstly enhanced in  the seventies by Serrin \cite{serrin} to get rigidity results  for overdetermined boundary value problems,    and afterwards  enlivened  to get symmetry and monotonicity properties of solutions to nonlinear elliptic equations, in particular by  Gidas-Ni-Nirenberg \cite{GNN79}, 
Berestycki-Nirenberg \cite{BN88, BN91},  Caffarelli-Gidas-Spruck \cite{CGS89}.

 The proof of Theorem \ref{t:serrin3} is based on a  reinvention  of the moving planes method in the context of measurable sets.  Alexandrov idea is that, if $\Omega$ has constant mean curvature, it must have a hyperplane of symmetry in every direction; this can be obtained starting from an arbitrary hyperplane,  moving it in a parallel way until an appropriate stopping time, and reflecting $\Omega$ about such hyperplane.  
The conclusion is then reached by the qualitative behaviour of the constant mean curvature equation, 
specifically using the strong maximum principle and Hopf  boundary point lemma. 
Our situation is completely different,  for many reasons. 
First, no connectedness assumption is made on $\Omega$, so that the proof cannot be obtained just by observing that the choice of the initial hyperplane is arbitrary,  but requires a new argument allowing to single out each ball and ``extract'' it from $\Omega$ once enough symmetries are detected. 
Second,  no smoothness information is available:
since the essential boundary does not admit a normal vector and is not locally a graph, 
all the   
the steps of the method loose their meaning. 
Third, even in cases when the boundary is locally a graph   and the contact with the reflected cap holds in classical sense,   
 no PDE holds around the contact point, but merely the $r$-criticality condition on the essential boundary. 
Thus we need to 
conceive new arguments, in particular to decipher why the movement can start and especially when it has to stop.   
In the smooth setting, this occurs either when the boundary and the reflected cap become tangent,  
or when the boundary and the moving plane meet orthogonally;
in the measurable setting,  these two situations
must be abandoned in favour of suitable notions of {\it away} or {\it close} contact. They are defined and handled
relying on the concept of Steiner-symmetric sets, which plays a crucial role similarly as in De Giorgi's proof of the isoperimetric theorem. 
A more detailed outline of the proof is given at the beginning of Section \ref{sec:proof}.


\bigskip

\section{Preliminaries}

In this section we discuss the main issues about  degeneracy, and we prepare the proof of Theorem \ref{t:serrin3}, by analyzing the structure of certain Steiner symmetric sets obtained by reflection.

\subsection{About   $r$-degeneracy} \label{sec:fat}

We start by showing some counterexamples of $r$-critical sets which escape from rigidity since they are degenerate.

\begin{examples}{\rm  Different kinds of bounded sets which, for some $r>0$, are critical but  degenerate:

\smallskip
{\it (i)  Small sets:} any measurable set $\Om$ with ${\rm diam } (\Om) \le r$.
  
 \smallskip
{\it  (ii) Unions of  small sets at large mutual distance:}   any measurable set  obtained as  the union of  a finite number of measurable sets $\Omega_j$, having  the same measure, 
${\rm diam } (\Om_j) \le r $ $\forall j$, and ${\rm dist} (\Omega _j, \Omega _ l) \geq r$ $\forall j \neq l$. 
  
 \smallskip 
 {\it (iii)  Unions of  spaced   small sets   at small mutual distance:}  
for $d= 2$ and any fixed $n \in \N$, given $r >0$ such that  $| | 1-e ^{i \frac{2\pi j}{n} }|-r| \ge \varepsilon>0$ $\forall j=1, \dots, n$, 
 any set obtained as the union of $n$ measurable sets $\Omega _j$, having the same measure, such that 
 $ \Omega _j\sq B_{\frac \varepsilon2} (e ^{i \frac{2\pi j}{n} })$ $\forall j = 1, \dots, n.$  
 This last example shows in particular that the connectedness of a $\frac{r}{2}$-neighbourhood of $\Omega$ is not sufficient to  avoid $r$-degeneracy.  
 
 }

\end{examples} 

\smallskip

Next we establish a measure-theoretic  sufficient condition for nondegeneracy (Proposition \ref{l:new}) which  is useful, in particular, to  deal with open sets and sets with finite perimeter   (Proposition \ref{l:fat}), 
and hence to deduce Corollaries \ref{t:serrin1} and \ref{t:serrin2}  from Theorem \ref{t:serrin3}. 

  Such condition is expressed in terms of the  
total variation measure $|D \chi _{B _ r (x)}|$, which 
 is given by
(see for instance \cite[page 117]{Ma12})
\begin{equation}\label{f:resphere}
|D \chi_{ B _ r (x)}| (E) = \mathcal H ^ { d-1} (\partial B _ r (x) \cap E) \quad \text{for any measurable set $E$};
\end{equation}
we are thus led back to handle the measure of spherical hypersurfaces considered by Cimmino.  

Hereafter and in the sequel, we denote by $\Omega ^ { (t)}$ the set of points $x \in \R ^ d$ at which $\Omega$ has  $d$-dimensional density equal $t$.

\begin{proposition}\label{l:new} Let $\Omega$ be a bounded measurable set, and let $r>0$.
Assume there exists $\varepsilon >0$ such that
\begin{equation}\label{f:totvar}
\inf_{x \in \mathcal U _\e (\partial ^* \Omega)}    |D \chi_{ B _ r (x)}| (\Omega ^ { (1)} )  >0 
\end{equation}
where $ \mathcal U _\e (\partial ^* \Omega)$ is the set of points at distance smaller than $\e$ from $\partial ^* \Omega$. 
Then $\Omega$ is not $r$-degenerate. 
\end{proposition}

\begin{remark} (i) Taking points of density $1$  in \eqref{f:totvar} is relevant in order to make the condition satisfied, for instance, by open sets deprived of a spherical hypersurface centred at a boundary point.

(ii) Condition \eqref{f:totvar} is not necessary for nondegeneracy. 
For instance,  consider in the complex plane
the union of the sets $\{z= \rho e^{i\theta} : 0<\rho <\frac 14, \theta \in (-\frac \pi4, \frac \pi4)\}$ and  $\{z= \rho e^{i\theta} : 1<\rho <3, \theta \in (-\frac \pi4, \frac \pi4)\}$. For $r=1$, the set is not degenerate. However,
 the infimum \eqref{f:totvar} vanishes, by taking points $x$ \EEE arbitrarily close to $0$ on the negative real axis.
\end{remark}

{\it Proof of Proposition \ref{l:new}.}  
 Set $\alpha:=\inf_{x \in \mathcal U _\e (\partial ^* \Omega)}    |D \chi_{ B _ r (x)}| (\Omega ^ { (1)} )$. 
Assume by contradiction that $\alpha >0$ but $\Omega$ is $r$-degenerate. Since $\Omega$ is bounded, we can find
two sequences $\{x_n\}, \{y_n\} \subset \partial ^* \Om$,  with $\lim _n \| x_n -y_n \|   =0$, such that
\begin{equation}\label{f:primula} 
\lim_{n \ra +\infty}\frac{\big |\Omega \cap ( B _ r (x_n) \Delta B _ r (y_n)) \big |  }{\| x_n -y_n \| }  = 0.
\end{equation}
Without loss of generality, we can assume that for every $n$ it holds 
$$ x_n, y_n \in {\mathcal U_{\frac \vps2}}(\partial ^* \Om) \quad \text{ and } \quad \|x_n-y_n\|\le \frac \vps2 \,.$$
This implies that, if $[x_n, y _n]$ is the closed segment with endpoints $x_n$ and $y _n$, we have \begin{equation}\label{f:segment} 
|D \chi_{ B _ r (x)}| (\Omega ^ { (1)} ) \ge  \alpha \qquad \forall x \in [x_n, y _ n]\,.
\end{equation} 

For every $x \in \R^d$, we denote $x'= (x^1, \dots, x^{d-1}) \in \R^{d-1}$, and we write $x = (x', x^d)$. 
We fix some $\delta_0 \in (0, \frac \vps4 \wedge r)$ such that 
\begin{equation}\label{f:schiscia}
{\mathcal H} ^{d-1} (\{ x \in \partial B_r (0): |x^d| \le \delta _0\}) \le \frac\alpha4.
\end{equation}
For convenience, we position our system of coordinates so that  $x_n = (0, \delta_n)$ and $y _n = (0, - \delta_n)$. 
Then, for $n$ large enough so that $\delta _n \leq \delta _0$, we get the following estimate:
$$\begin{array}{ll}
& \displaystyle \frac{\big |\Omega \cap ( B _ r (x_n) \Delta B _ r (y_n)) \big |  }{\| x_n -y _n \| }   \displaystyle = \frac{1}{2 \delta_n} \int_{B _ r (x_n) \Delta B _ r (y _n)} \chi_ \Om  (x) \,  dx \\ 
\noalign{\bigskip}
 \ge & \displaystyle \frac{1}{2 \delta_n}  \int_{-\delta_n }^{\delta _n } \Big [ \int_{\|x'\| \le \sqrt{r^2-\delta_0^2}} \chi_{\Om ^{(1)}} (x', \sqrt{r^2-\|x'\|^2}+s) dx' \Big ]ds 
 \\ \noalign{\bigskip}
 + & \displaystyle \frac{1}{2 \delta_n}  \int_{-\delta _n }^{\delta _n } \Big [ \int_{\|x'\| \le \sqrt{r^2-\delta_0^2}} \chi_{\Om ^{(1)}}  (x', -\sqrt{r^2-\|x'\|^2}+s) dx' \Big ]ds
 \\ \noalign{\bigskip}
 \ge & \displaystyle \frac{1}{2 \delta_n}  \frac {\sqrt{r^2-\delta_0^2}}{r}  \int_{-\delta_n }^{\delta_n} \Big [ \int_{\|x'\| \le \sqrt{r^2-\delta_0^2}} \chi_{\Om ^{(1)} }(x', \sqrt{r^2-\|x'\|^2}+s) \frac {r} {\sqrt{r^2-\|x'\|^2}} dx' \Big ]ds
 \\ \noalign{\bigskip}
+ & \displaystyle  \frac{1}{2 \delta_n}  \frac {\sqrt{r^2-\delta_0^2}}{r}  \int_{-\delta_n }^{\delta_n} \Big [ \int_{\|x'\| \le \sqrt{r^2-\delta_0^2}} \chi_{\Om ^{(1)}}   (x', -\sqrt{r^2-\|x'\|^2}+s) \frac {r} {\sqrt{r^2-\|x'\|^2}} dx' \Big ]ds
 \\ \noalign{\bigskip}
\ge & \displaystyle
\frac {\sqrt{r^2-\delta_0^2}}{r} \frac \alpha2\,,
 \end{array}
$$ 
where the last inequality follows from \eqref{f:segment} and \eqref{f:schiscia}. This contradicts \eqref{f:primula} and achieves the proof. \qed  

\bigskip

\bigskip 

\begin{proposition}\label{l:fat}   
Let $\Omega$ be either a bounded open set or a bounded set of finite perimeter, and let $\{\Omega_i\} _i$ denote the family  respectively of its
connected or indecomposable components. 
Then $\Omega$ is  not $r$-degenerate   for any $r < \inf_i \{ {\rm diam} (\Omega _i ) \}$. 
 \end{proposition}

 \begin{remark} While sets $\Omega$  with a single component are   not $r$-degenerate   if and only if $r < {\rm diam} (\Omega )$, for multiply connected domains the condition $r < \inf_i \{ {\rm diam} (\Omega _i ) \}$ is sufficient but clearly not necessary   to avoid $r$-degeneracy: for instance, the disjoint union of two balls of radius $R$ can be  not $r$-degenerate  also for radii $r\geq 2 R$ provided the balls are close enough.  \end{remark}
 
 \begin{remark} In the light of Proposition \ref{l:fat}, it is natural to ask if there exists some notion of ``connectedness''   avoiding $r$-degeneracy also for arbitrary 
 measurable sets. To the best of our knowledge, the unique kind of such a general notion   for Borel sets  $\Omega$ has been proposed in  \cite{CCDM} under the name of {\it essential connectedness}, and amounts to ask that $\mathcal H ^ { d-1} (\Omega ^  {(1)}  \cap \partial ^* \Omega _+ \cap \partial ^ * \Omega _ -) >0$ for any nontrivial Borel partition $(\Omega _+ , \Omega _-)$ of $\Omega$.   However, if $\Omega$ has not finite perimeter,  relying on the possible lack of semicontinuity of the map $t \mapsto \mathcal H ^ { d-1} (\Omega \cap B _ {r} ( x_t)) )$ as $x _ t \to x \in \overline {\partial ^* \Omega}$,  it is possible to construct examples of Borel sets $\Omega$ which are essentially connected but  degenerate   for some $r < {\rm diam} (\Omega)$.

 \end{remark}

{\it Proof of Proposition \ref{l:fat}}.  We focus our attention on the case of finite perimeter sets, 
since for open sets  the proof can be obtained   in a similar way. Working component by component,  we are reduced to prove that, if $\Omega$ is indecomposable,  then it is not $r$-degenerate for any $r < {\rm diam} (\Omega)$. 
To that aim, it is enough to prove that it satisfies, for some $\e>0$, condition \eqref{f:totvar} in  Proposition \ref{l:new}. 
 Assume by contradiction this is not the case. 
 Then, in view of \eqref{f:resphere}, it would be possible to find a sequence $\{x_n\}$ converging to a point $\overline x \in \overline {\partial ^ * \Omega}$ such that
 \begin{equation}\label{f:degsphere}
  \lim _n \mathcal H ^ { d-1} (\partial B _ r (x_n) \cap \Omega ^ { (1)} ) = 0\,.
  \end{equation}  
Up to subsequences, we denote by $A$ and $C$ the limit in $L ^ 1$ respectively of the characteristic functions of the  sets
 $A_n := \Omega \cap B _ r (x_n)$  and $C_n := \Omega \setminus B _ r  (x_n)$. 
 We are going to show that they provide a nontrivial partition of $\Omega$  such that
 ${\rm Per} (\Omega) = {\rm Per} (A) + {\rm Per} (C)$.  
 We have
 $$\begin{cases} {\rm Per} (A_n) = {\rm Per} (\Omega, B _ r (x_n))  + \mathcal H ^ { d-1} (\partial B _ r (x_n) \cap  \partial ^ * A_n  \cap \Om^{(\frac 12)})  +  \mathcal H ^ { d-1} (\partial B _ r (x_n) \cap \Omega ^ { (1)} )   & 
 \\ 
 \noalign{\bigskip} 
 {\rm Per} (C_n) = {\rm Per} (\Omega, \R ^ d \setminus \overline {B _ r} (x_n))  + \mathcal H ^ { d-1} (\partial B _ r (x_n) \cap  \partial ^ *C_n \cap \Om^{(\frac 12)})  +  \mathcal H ^ { d-1} (\partial B _ r (x_n) \cap \Omega ^ { (1)} )
 \end{cases} 
  $$ 
Since   perimeter is lower semicontinuous with respect to $L ^ 1$-convergence, and we have  ${\mathcal H}^{d-1} ( \partial ^ *A_n   \cap  \partial ^ *C_n   \cap \Om^{(\frac 12)})=0$,  by   passing to the limit in the two relations above  and summing, we get 
  $${\rm Per} (A)+ {\rm Per} (C) \leq {\rm Per} (\Om) + 2  \lim _n \mathcal H ^ { d-1} (\partial B _ r (x_n) \cap \Omega ^ { (1)} ).$$
The conclusion follows by condition  \eqref{f:degsphere}. \qed

\smallskip

\bigskip

\subsection{About some Steiner symmetric sets obtained by reflection} 

A measurable set $\omega$  is {\it Steiner symmetric} about a hyperplane $H$ with unit normal $\nu$ if 
the following equality holds 
as an equivalence between Lebesgue measurable sets: 
$$
\omega =  \Big \{ x \in \R ^d \ :\ x = z + t \nu \, , \ z \in H  \,  , \ |t| < \frac{ 1}{2} \mathcal H ^ 1 \big (
 \omega \cap \big \{ z + t \nu  : t \in \R \big \} \big ) 
   \Big \}
 \,.$$ 
 In Proposition \ref{l:reflection} below, we focus our attention on a special kind of Steiner symmetric sets obtained by reflection, 
 that we shall need to handle in the proof of Theorem \ref{t:serrin3}. 
 
 \smallskip 
To that aim and in the sequel, we shall make repeatedly use of the following elementary observation:
given a measurable subset $\omega$ of $\R^d$, it holds 
\begin{equation}\label{f:top}
\omega ^ {(1)} \setminus  \overline{ \partial ^ * \omega } =  {\rm int} (\omega ^ { (1)} ) \,, \qquad \omega ^ {(0)} \setminus  \overline{ \partial ^ * \omega } =  {\rm int} (\omega ^ { (0)} )\,.
\end{equation}
In particular,  $\R ^d$ can be decomposed as a disjoint union, 
\begin{equation}\label{f:decompose} 
\R ^d = {\rm int} ( \omega ^ {(1)})   \sqcup {\rm int}  (\omega ^ {(0)}) \sqcup \overline{ \partial ^* \omega}\,. 
\end{equation}

Let us prove the first equality in \eqref{f:top}, the second one being analogous. The inclusion $\supseteq$ is immediate. Viceversa, let $x \in \omega ^ {(1)} \setminus  \overline{ \partial ^ * \omega }$, and let $U$ be an open neighbourhood of $x$ which does not meet $\partial ^ *\omega$.  
Let us prove that   $U \subset \omega ^ { (1)}$. 
By Federer's Theorem, $\omega$ is of finite perimeter in $U$. 
Then, by the relative isoperimetric inequality, 
$\min \{  | \omega^c \cap U | , |\omega \cap U | \} = 0$. But it cannot be $| \omega\cap U | = 0$, because $x \in \omega^ {(1)}$. 
Hence $|\omega^c \cap U|= 0$. Then, $U$ cannot contain any point of density $0$ for $\omega$, since such point would be of density $1$ for $\omega ^ c$, against $|U\cap \omega^c|= 0$. 
Recalling that $U$ does not meet $\partial ^* \omega$, we conclude that $U \subset \omega ^ { (1)}$.  
\bigskip

\begin{proposition}\label{l:reflection} Let $H$ be a hyperplane with unit normal $\nu$, and
 let $\omega$ be a bounded measurable set contained into $H_- = \{H + t \nu \, :\, t \leq0 \}$ such that
  \begin{equation}\label{f:singleton} \forall z\in H\,, \quad \overline{\partial ^* \omega}  \cap \big \{ z  + t \nu \ :\ t <0 \big \}  \text{  is empty or a singleton\,, }\end{equation} 
so that $\omega$ can be viewed as the subgraph of the function $g:H \ra \R_-$ defined by  $g(z)=0$ if the intersection  in \eqref{f:singleton} is empty and 
$g(z)= t$ if such intersection is $ z  + t \nu $.   The following properties hold: 

\smallskip 
\begin{itemize}
\item[(i)]  $|\overline{\partial ^* \omega} | = 0$;

\smallskip
\item[(ii)] $\omega$ is essentially open;

\smallskip
\item[(iii)]  the union of $\omega$ and its reflection about $H$ is Steiner-symmetric about $H$;  

\smallskip
\item[(iv)]  the function $g$ is continuous. 
\end{itemize} 
\end{proposition}

\proof Statement (i) is an immediate consequence of the assumption \eqref{f:singleton} and Fubini Theorem. 
To obtain statement (ii), it is enough to show that $\omega ^ { (1)}$ is essentially open. Such  property follows from statement (i), after applying \eqref{f:top}. 
To prove statement (iii), it is enough to show that 
the union of $\omega^ {(1)} $ and its reflection about $H$ is Steiner-symmetric about $H$. To that aim, let us fix  $z \in H$ such that $\overline {\partial ^* \omega} \cap \big \{ z  + t \nu \, :\, t  <0  \big \}= \{p\} $, and let us show that the open segment $(p, z)$  is contained into $\omega ^ { (1)}$. Recalling the decomposition \eqref{f:decompose} and assumption \eqref{f:singleton}, we infer that the open segment $(p, z)$ is entirely contained either in ${\rm int} (\omega ^ { (1)})$ or in ${\rm int } (\omega ^ { (0)})$. In the first case we are done. It remains to exclude that it is entirely contained in ${\rm int} ( \omega ^ { (0)})$. Assume by contradiction this is the case. We observe that, by the first equality in \eqref{f:top}, since $p  \in \overline {\partial ^* \omega}$, $p $ is the limit of a sequence of points $\{ p _n\} \subset \omega ^ { (1)}$. Since we are assuming that   the open segment $(p, z)$ is entirely contained in ${\rm int}( \omega ^ { (0)})$, the points $p _n$ do not belong to such segment, so that they belong to straight lines of the form $\{ z_ n + t  \nu\, : \, t \in \R \}$,  for a sequence of points $\{z_n\}\subset H $ converging to $z$. But then some of these straight lines would necessarily contain at least two points of $\overline {\partial ^* \omega}$ 
(otherwise the segments $(p _n, z_n)$ would be entirely contained into ${\rm int}(\omega ^ { (1)})$ and could not converge to  $(p, z)$ which is entirely contained into ${\rm int}(\omega ^ { (0)})$).

Let $g$ be the function defined via \eqref{f:singleton} as in the statement, so that $\omega$ can be viewed as the subgraph of $g$. To show the continuity of $g$ at a fixed point   $z_0 \in H$, we consider separately the cases $g ( z_0) = 0$ and $g ( z_0) <0$. If $g ( z _0)= 0$, 
we have to prove that, for any 
sequence $\{z_n\} \subset H$ converging to $z_0$,  the sequence $\{g ( z_n)\}$ converges to $0$. Up to a subsequence, we may assume $g ( z_n ) \to \lambda$, with $\lambda \leq 0$. If $\lambda < 0$, for $n$ large enough we have $g ( z_n) <0$, which by definition  of $g$ means that $z _n +  g ( z_n ) \nu \in \overline {\partial ^* \omega}$;  passing to the limit in the last relation, we get $z_0 + \lambda \nu \in  \overline {\partial ^* \omega}$, against $g ( z_0) = 0 $. 

Assume now $g ( z_0) < 0$, and let $\{z_n\} \subset H$ be any sequence converging to $z_0$. We may assume  that
 $g ( z_n) \to \lambda $, with $\lambda \leq 0$. If $\lambda < 0$, we get as above $z_0 + \lambda \nu \in 
\overline {\partial ^* \omega}$; by \eqref{f:singleton},  we conclude that $g ( z_0) = \lambda$.
It remains to show that the case $\lambda = 0 $ cannot occur. 
Assume $\lambda  = 0$. We consider the open segment $S:= (z_0 + g ( z_0)\nu , z_0) $. By \eqref{f:singleton}, $S \cap \overline {\partial ^* \omega} = \emptyset$, and hence $S$ is entirely contained either into ${\rm int} (\omega ^{(1)})$ or into ${\rm int} (\omega ^ { (0)})$.  If $S \subseteq {\rm int} (\omega ^{(1)})$, 
we pick a point $x_0 \in S$ and a small ball $B _ \e ( x_0) \subset \omega ^{(1)}$.  For every $x \in B _ \e ( x_0)$, denoting by $z_x$ its projection onto $H$ (in particular, $z_{x_0} = z_0$),   by statement (iii) we have that the segment $(x, z_x)$ lies into $\omega ^ {(1)}$.  This property leads to a contradiction, as it implies 
on one hand that   $z_0 \in {\rm int} (\omega ^{(1)}) $ and on the other hand that that  $g ( z_n) < 0 $ for $n$ large enough, 
which in turn gives $z_0 \in \overline  {\partial ^* \omega}$ (passing to the limit in the relation $z _n + g ( z_n ) \nu \in \overline  {\partial ^* \omega}$). 
If $S\subseteq {\rm int} (\omega ^{(0)})$, 
we can pick a point $x_0 \in S$ and a small ball $B _ \e ( x_0) \subset \omega ^{(0)}$. 
This contradicts statement (iii) and the fact that, since the point $z_0 + g ( z_0)\nu$ belongs to $\overline  {\partial ^* \omega}$, it is the limit of a sequence of points of density $1$ for $\omega$.  \qed

\bigskip

\section{Proof of Theorem \ref{t:serrin3}}\label{sec:proof}

\smallskip
{\bf Outline of the proof.} We observe first of all that the equality \eqref{f:hyp3} continues to hold at every point $x \in \overline{\partial ^ * \Omega}$. Then we 
fix a direction $\nu \in S ^ {d-1}$, and  we consider an initial hyperplane $H _ 0$ with unit normal $\nu$, not intersecting $\overline {\partial ^* \Omega}$.  Such an initial hyperplane exists because, since $\Omega$ has finite measure and is $r$-critical, it is necessarily bounded. We start moving $H _ 0$ in the direction of its normal $\nu$ to new positions, 
so that at a certain moment of the process it starts intersecting $\overline {\partial ^* \Omega}$. 
We continue the movement in direction $\nu$, 
 and we denote by $H _t$ the hyperplanes thus obtained. We set: 
 $$
   \begin{array}{ll} 
& H_t ^- := \text { the closed halfspace determined by $H_t$  containing $H_0$} 
 \\ \noalign{\medskip}
 &H_t ^+  := \text { the closed halfspace determined by $H_t$  not containing $H_0$} 
\\ \noalign{\medskip}
& \Omega_ t  := \Omega \cap H _ t ^ {-}
 \\ \noalign{\medskip}
& \reflext := \text { the reflection of $\Omega _t $ about $H _t$}. 
\end{array}
$$

\smallskip

$\bullet$ We say that {\it symmetric inclusion} holds at $t$ if
 \begin{equation}\label{f:moving} 
\reflext \subset \Omega \qquad \text{ and } \qquad \Omega _t   \cup \reflext   \text{ is Steiner symmetric about $H _t$\,.} \end{equation} 
 
$\bullet$ We say that  symmetric inclusion occurs at $t$ if  {\it with away contact} if \eqref{f:moving} holds and there exists an ``away contact point'', namely a point
 \begin{equation}\label{f:touching}
 p' \in  \big [ \overline {\partial ^* \reflext } \cap \overline {\partial ^*  \Omega}\big ] \setminus H _ t  \, .
  \end{equation}
when \eqref{f:moving} holds but \eqref{f:touching} is false, we say that symmetric inclusion at $t$ holds  {\it without away contact}. 

\smallskip
$\bullet$ We say that symmetric inclusion occurs at $t$ {\it with close contact}  if \eqref{f:moving} holds and there exists a ``close contact point'', namely a point  \begin{equation}\label{f:touching2}
H _ t \ni  q = \lim _n q _{1,n} = \lim _n q  _ {2,n},   \quad  q _{i, n}  \in  \overline{\partial ^* \Omega}\cap \{ q + t \nu\, :\, t \in \R \}, \quad   q _ {1,n} \neq  q _{2,n}\,,  \end{equation}

 (where one among $q _{1,n} $ and $q  _ {2,n}$ will always happen to belong to $H _ t ^ +$, while the other one may fall in $H _ t ^ +$ as well as in  $H _ t ^-$.)   
Notice that symmetric inclusion can occur at the same $t$ with both away contact and close contact. 

\smallskip
 The statement will be obtained in the following steps, which are carried over separately in the next subsections.  
 
 \medskip
 {\bf Step 1 (start)} 
 
 There exists $\e >0$ such that, for every $t \in [0, \e)$, symmetric inclusion holds.  

 \medskip
 {\bf Step 2 (the stopping time: no close contact without away contact)} 
 
 Setting
 $$T:= \sup \Big \{ t >0 \ :\ \text{ for all  $s \in [0, t)$,  symmetric inclusion occurs without away contact} \Big \} \, , $$ 
we have $T < + \infty$, and symmetric inclusion occurs at $T$ with away or with close contact. 
But we are able to rule out the case of close contact without away contact, so necessarily at $t = T$ we are in the situation of away contact.

 \medskip
 {\bf Step 3 (decomposition of $\Omega$ into symmetric and non-symmetric part)} 
 
 We show that $\Omega$ can be decomposed as 
$$
\Omega= \Omega^s\sqcup \Omega ^{ns}\, , 
$$ 
where  $\Omega ^ s$ is an {\it open set}  representing the  Steiner symmetric part of $\Omega$, given by 
$$
\Omega ^ s := \bigcup \Big \{ (p, p') \ :\ p' \text{ is an away contact point, \ $p$ is its symmetric about $H _ T$} \Big \} \,, 
$$
$(p, p')$ being the open segment with endpoints $p$ and $p'$, 
and $\Omega ^ { ns}:= \Omega \setminus \Omega ^ {s}$ represents the non-symmetric part. Moreover, denoting by $ \Omega ^ s _ i$  
the open connected components of $\Omega ^ s$, 
we  prove that:
\begin{eqnarray}
& \text{ 
 $\overline{\partial ^ * \Omega^ s_i} \cap (H_ T^ \pm  \setminus H _ T)$  are  connected sets;} 
  & \label{f:cc} 
\\ \noalign{\medskip} 
& \overline {\partial ^* \Omega ^s} \cap \overline {\partial ^* \Omega ^{ns}} \subset H _ T\,.
 & \label{f:studiobordi}
\end{eqnarray}

 \medskip

 \medskip
 
 {\bf Step 4 (conclusion)}   
 We show that the open connected components of $\Omega ^ s$  are balls
of the same radius $R> r/2$, lying at distance larger than or equal to $r$, while the set $\Omega ^ { ns}$ is Lebesgue negligible.  

 \bigskip
\subsection{ Proof of Step 1.}   The proof is based on the following lemma.

\begin{lemma}[no converging pairs]\label{l:nopairs} 
Let $\Omega\subset \R ^d $ be a measurable set which is $r$-critical and  not $r$-degenerate.  Assume that $\Omega $ is contained  into
$ H _0 ^+:=\{ z + t \nu\, :\, z \in H _0\, , \ t \geq 0 \}$, 
 $H_0$ being a hyperplane  with unit normal $\nu$. Then there cannot exist two sequences of  points 
$\{ p_{1,n}\}$, $\{p_{2,n}\}$ in $\overline {\partial ^* \Omega}\cap H _0 ^+$  
  which  for every fixed $n$ are  distinct, with  the same projection onto $H_0$, and 
at infinitesimal distance from $H _ 0$ as $n \to + \infty$. 
\end{lemma}

%

\proof We argue by contradiction.
 Setting $t_{i, n} : =  {\rm dist} ( p  _{i,n} , H_0)$, we can assume up to a subsequence  that
$t _ {1, n} > t _{ 2 ,n }$ for every $n$.
We are going to show that
\begin{equation}\label{f:liminf} \liminf _{n \to + \infty} \frac{|\Omega \cap B _ r (p_{1,n}) | - |\Omega \cap B _ r (p_{2,n}) | }
 {t _{ 1 ,n}- t _{ 2 ,n} } >0 \ ,
\end{equation} 
against the fact that $\Omega$ is $r$-critical. 
We have
\begin{equation}\label{f:minus} |\Omega \cap B _ r (p_{1,n}) | - |\Omega \cap B _ r (p_{2,n}) | = |\Omega \cap (B _ r (p_{1,n})  \setminus B _ r (p_{2,n}) ) | - | \Omega \cap (B _ r (p_{2,n})  \setminus B _ r (p_{1,n})) | \,.
\end{equation}
Since $\Omega$ is  not $r$-degenerate, there exists a positive constant $C$ such that
\begin{equation}\label{f:plus} \frac{ |\Omega \cap (B _ r (p_{1,n})  \setminus B _ r (p_{2,n}))  | + | \Omega \cap (B _ r (p_{2,n})  \setminus B _ r (p_{1,n}) ) | }{t _{1, n } - t _ {2 ,n } }   \geq C \,.
\end{equation}
In view of \eqref{f:minus} and \eqref{f:plus},  the inequality \eqref{f:liminf} holds true provided
\begin{equation}\label{f:goal}
 \frac{ | \Omega \cap (B _ r (p_{2,n})  \setminus B _ r (p_{1,n})) |} {t _{ 1 ,n} - t _{ 2 ,n}} \leq \frac{C}{4} \,.
\end{equation} 
In turn, by the inclusion $\Omega \subset H _0 ^+$, the inequality \eqref{f:goal} is satisfied as soon as 
\begin{equation}\label{f:goal}
 \frac{ | H_0 ^+\cap (B _ r (p_{2,n})  \setminus B _ r (p_{1,n})) |} {t _{ 1 ,n} - t _{ 2 ,n}} \leq \frac{C}{4} \,.
\end{equation} 
Such inequality follows from elementary geometric arguments. Indeed,  for every fixed $n$, the set 
$H_0^+ \cap (B _ r (p_{2,n})  \setminus B _ r (p_{1,n}))$   has volume not larger  than the region $D_n$ obtained as the difference between two right cylinders having the same axis, 
given by the straight line orthogonal to $H _ 0$ through the common projection $z_n$ of $p  _{1,n}$ and $p  _{2,n}$  onto $H_0$, 
  the same height equal to $t_{2,n} + (1/2) ( t_{1,n}  - t _{2,n})$,  and as bases the $(d-1)$-dimensional balls contained into $H_0$, with center at  $z_n$
and radii respectively  equal to 
$( r ^ 2 - t_{2,n} ^ 2  ) ^ {1/2}$  and $ ( r ^ 2 - t_{1,n} ^ 2 ) ^ {1/2}$. 
Hence, to get \eqref{f:goal} it is enough to show that
${|D_n | } = o ( t _{ 1 ,n} - t _{ 2 ,n} )$.
This is readily checked since, 
setting $ \gamma _n :=(  t _{1,n} - t _{2,n} )/ 2$, we have
$$
\begin{array}{ll} |D_n |  & \displaystyle = \omega _{d-1} \big  (  ( r ^ 2 - t_{2,n} ^ 2  ) ^ {\frac{d-1}{2} }  - ( r ^ 2 - t_{1,n} ^ 2 ) ^ {\frac{d-1}{2} }   \big  ) \big (  t_{2,n} + \frac{1}{2} ( t_{1,n}  - t _{2,n})  \big ) \\ \noalign{\medskip} 
& \displaystyle \sim  2( d-1)   \omega _{d-1}  r ^ {d-3} 
( t_{2,n} \gamma_n +  \gamma _n ^ 2  ) \big (  t_{2,n} + \gamma _n  \big )
\,.  \end{array} $$
\qed
\bigskip

Assume the claim in Step 1 false. Then, at least one of the following assertions holds:
 
\begin{itemize}
\item[ (i)]   $ \exists \{t _n\}\to 0$ such that $\forall n$  $\Omega _{t_n}  \cup  \mathcal R _{t_n} $ is {\it not} Steiner symmetric about $H _ { t_n}$; 
 
 \smallskip
\item[(ii)]   $\exists \{t _n\} \to 0$ such that $\forall n$  $|\mathcal R _{t_n}  \setminus \Omega| >0$. 
  \end{itemize}
  
  \smallskip 
  In case (i), for every $n$ we can apply Proposition \ref{l:reflection}   with $H = H _ {t_n}$ and $\omega = \Omega ^- _{t_n}$ to infer that, for some $z_n \in H _ { t _n}$,  the set
   $\overline{\partial ^* \Omega  _{t_n} }  \cap  \{ z_n  + t \nu \ :\ t < 0  \}$ contains at least two distinct points.
Then $\overline{\partial ^* \Omega }\cap H _0^ + $ contains two sequences  of points $\{p _ {1 ,n}\}$, $\{p_{2,n}\}$  which for every $n$ are distinct, with the same projection onto $H _0$, and at infinitesimal distance from $H _ 0$ as $n \to + \infty$, against Lemma \ref{l:nopairs}. 

\smallskip
In case (ii), we may assume that  $\Omega _{t_n}  \cup  \mathcal R _{t_n} $ 
is Steiner symmetric about $H _ { t_n}$
For every $n$ let $y' _n \in \Omega^ {(0)} \cap \mathcal R _{t_n} ^ { (1)} $, and let  
$z _n$ be the orthogonal projection of  $y' _n$ on $H _ { t _n}$. 

If on the segment $(z_n, y' _n]$ there is some point in $\overline{\partial ^ * \Omega}$, 
since $\Omega _{t_n}  \cup  \mathcal R _{t_n} $  is Steiner symmetric about $H _ { t_n}$, 
we would have a pair  of distinct points belonging to $\overline{\partial ^* \Omega }\cap H _0^+ $, 
with the same projection on $H _0$, and infinitesimal distance from $H _0$, against Lemma \ref{l:nopairs}. 

If on the segment $(z_n, y' _n]$ there is no point in $\overline{\partial ^ * \Omega}$, 
invoking  \eqref{f:decompose} and recalling that $y' _n \in \Omega^ {(0)}$, we infer that the whole segment $(z_n, y '_n ]$
 is contained into  ${\rm int} (\Omega ^ { (0)})$. On the other hand, since $y' _n \in \mathcal R _{t_n} ^ { (1)}$, 
 denoting by $y _n$ the reflection of $y' _n$ about $H_{t_n}$, we have that the whole segment $[y_n, z_n)$ 
 is contained into ${\rm int} (\Omega ^ { (1)})$. We conclude that the point $z _n$ belongs to $\overline{\partial ^ * \Omega}$.  
Then, by arguing in the same way as in the last part of the proof of Proposition \ref{l:reflection}, 
it would be possible to find some straight line of the form $ \{ \widetilde z _n  + t \nu \, : \, t \in \R \}$, with $\widetilde z _n\in H _ {t_n}$ arbitrarily close to $z _n$, containing at least two points of $\overline{\partial ^ * \Omega}\cap (H _0) _+$.
Again, this would contradict Lemma \ref{l:nopairs}.

   \bigskip
\subsection{Proof of Step 2.}   
 Since $\Omega$ is bounded, we have $T < + \infty$. Then 
the proof of Step 2 is obtained by showing the following claims:
\smallskip

$\bullet$ {\it Claim 2a.  Symmetric inclusion  holds at $T$
 with away contact or with close contact.}
{\smallskip}

$\bullet$  {\it Claim 2b. Symmetric inclusion cannot hold with close contact and no away contact.} 

\bigskip
{\it Proof of Claim 2a}. 
Symmetric inclusion clearly continues to hold at $T$. 
 Moreover, by definition of $T$, at least one of the following assertions is true:

 \smallskip
 \begin{itemize}
\item [(i)] $\exists \{ t_n \} \to T ^+$ such that  $\forall n$ symmetric inclusion with away contact holds at $t _n$;  

\smallskip
 \item [(ii)] $\exists \{ t_n \} \to T ^+$ such that $\forall n$ symmetric inclusion does not hold at $t _n$.  
  \end{itemize}

 \medskip
In case (i),  for every $n$ there exists an away contact point at $t _n$, namely a point 
 $p'_n \in  \big [ \overline {\partial ^* \mathcal R _ {t_n} } \cap \overline {\partial ^*  \Omega}\big ] \setminus H _{ t_n}$. 
 Up to a subsequence, denote by $p'$ the limit of $p'_n$. Two cases may occur. If 
 $p' \not \in H _ T$, then  $p'$ is an away contact point at $T$. 
 If  $p' \in H _ T$,  
 denoting by $p _n$ the symmetric of $p '_n$ about $H _ { t_n}$, taking $q_{1,n} = p _n$ and $q _{2,n} = p' _n$ in 
 \eqref{f:touching2}, we see that $p'$ is a close contact point.

 \smallskip
To deal with case (ii), we point out the validity of the following 

  \smallskip
 {\it Away inclusion property: If symmetric inclusion occurs without away contact at $T$, for every $\delta >0$, there exists $s _\delta >0$ such that, for every $s \in [0, s _\delta]$ the set
$$
 U _{T- \delta} ^ s :=   \Big \{ x + ( 2 \delta + 2 s ) \nu \ :\ x \in \mathcal R _{T- \delta}  \Big \} 
$$
 is contained into $\Omega$. }

 \smallskip
 The away inclusion property can be easily proved by contradiction. If it was false, we could find an infinitesimal sequence $\{s_n\}$ of positive numbers, and a sequence  of points $\{x'_n\}$  of density $1$ for $U _{T- \delta} ^ {s_n}$  but of density $0$ for $\Omega$. Up to a subsequence, there exists $x' := \lim _n x' _n$. 
 By construction, we have  
 $x' \in \{ x + 2 \delta \nu : x \in \overline { \partial ^ *  \mathcal R _{T- \delta}} \}  \subset  \overline { \partial ^ * \mathcal R _T }$. 
But, since we are assuming that symmetric inclusion occurs without away contact at $T$, it is readily checked that
 $\overline { \partial ^ *  \mathcal R _T} \subseteq {\rm int} (\Omega ^ { (1)} )$. 
 Then $x' \in  {\rm int} (\Omega ^ { (1)} )$, against the fact that $x' _n$ are points of density $0$ for $\Omega$. 
 
 \smallskip

  Now, going back to case (ii), we can assume that symmetric inclusion occurs at $T$ without away contact (otherwise Claim 2a.\ holds for free). 
  Then, in view of the away inclusion property, the failure
of symmetric inclusion at $t _n$ implies that, for every $n$ and every $\delta>0$,
  there exist at least two distinct points in $\overline{ \partial ^* \Omega}$, say $q_{1, \delta, n}$ and $q _{2, \delta, n}$,  
   which have the same orthogonal projection onto $H _ T$ and have distance less than $\delta$ from $H _ T$. 
By the arbitrariness of $\delta >0$, we can choose a diagonal sequence, 
and passing to the limit we get a close contact point at $T$ according to definition \eqref{f:touching2}. 

 \medskip 
 {\it Proof of Claim 2b.} 
Assume by contradiction that symmetric inclusion holds at $T$ with close contact. We are going to  contradict \eqref{f:hyp2} by showing that, if
$\{q_{1,n}\}$ and
  $\{q _{2,n}\}$ are sequences converging to a point $q\in H _ T$ as in \eqref{f:touching2}, it holds 
\begin{equation}\label{f:liminf2} \liminf _{n \to + \infty} \frac{|\Omega \cap B _ r (q_{1,n}) | - |\Omega \cap B _ r (q_{2,n}) | }
{ \| q _{ 1 ,n}- q _{ 2 ,n} \|}   >0 \,.
\end{equation}

We have 
\begin{equation}\label{f:differenza} 
\begin{array}{ll}
& |\Omega \cap B _ r (q_{1,n}) | - |\Omega \cap B _ r (q_{2,n}) | = \\ \noalign{\bigskip}
 & |\Omega \cap (B _ r (q_{1,n})  \setminus B _ r (q_{2,n}) ) | - | \Omega \cap (B _ r (q_{2,n})  \setminus B _ r (q_{1,n})) | \,.
\end{array}
\end{equation}

In order to estimate the two terms at the r.h.s.\ of \eqref{f:differenza}, we fix $\delta >0$ (to be chosen later), and  we let $s _n >0$ be such that $H _ { T + s_n  }$ contains the midpoint of the segment $(q_{1,n}, q_{2,n})$. Up to working with $n$ large enough, since $q_{1,n}$ and  $q_{2,n}$ converge to a point of $H _ T$, thanks to the away inclusion property we can assume that  \begin{equation}\label{f:contenimento} 
 U _{T- \delta} ^ {s_n}  \subset \Omega \,.
 \end{equation}
 Hence,
\begin{equation}\label{f:destra} 
\begin{array}{ll}  
| \Omega \cap (B _ r (q_{1,n})  \setminus B _ r (q_{2,n})) |   & =  
 | U _{T- \delta} ^ { s _n} \cap (B _ r (q_{1,n})  \setminus B _ r (q_{2 ,n})) |   \\ \noalign{\bigskip} 
 & + \ | ( \Omega \setminus  U _{T- \delta} ^ { s _n})  \cap (B _ r (q_{1,n})  \setminus B _ r (q_{2 ,n})) |\,. \end{array}
\end{equation}
On the other hand, 
\begin{equation}\label{f:sinistra} \begin{array}{ll} 
| \Omega \cap (B _ r (q_{2,n})  \setminus B _ r (q_{1,n})) |  \displaystyle & =   | \Omega _{T - \delta}  \cap (B _ r (q_{2,n})  \setminus B _ r (q_{1,n})) |    \\ \noalign{\bigskip} 
& +\  | \Omega \cap  (     H _{T+ s_n} \oplus B _ {\delta + s _n}  ) \cap (B _ r (q_{2,n})  \setminus B _ r (q_{1,n}) ) |
 | 
\end{array}
\end{equation}
Here and below, we denote by $H \oplus B _ R (0)$ the strip given by points of $\R ^d$ with distance less than $R$ from a hyperplane $H$.

The two sets
$  \Omega _{T - \delta}  \cap (B _ r (q_{2,n})  \setminus B _ r (q_{1,n}))$ and 
$U _{T- \delta} ^ { s _n} \cap (B _ r (q_{1,n})  \setminus B _ r (q_{2 ,n}))$
have the same measure as they are symmetric about the hyperplane $H _ { T + s_n}$.
Therefore, by subtracting \eqref{f:sinistra} from \eqref{f:destra}, and recalling \eqref{f:differenza}, we obtain
\begin{equation}\label{f:decomponi} 
  |\Omega \cap B _ r (q_{1,n}) | -  |\Omega \cap B _ r (q_{2,n}) | = I _n - J _n\, , 
  \end{equation}
with 
$$\begin{cases} 
I _n  := 
   |( \Omega \setminus  U _{T- \delta} ^ { s _n})  \cap (B _ r (q_{1,n})  \setminus B _ r (q_{2 ,n})) |  & 
   \\  \noalign{\bigskip}
   J _n := 
    | \Omega \cap  (     H _{T+ s_n} \oplus B _ {\delta + s _n}  ) \cap (B _ r (q_{2,n})  \setminus B _ r (q_{1,n}) ) |   \,. &
    \end{cases}
$$ We  are now going to estimate $J _n$ from above and $I _n$ from below. We
have 
 $$ J _n \leq  |     (H _{T+ s_n} \oplus B _ {\delta + s _n}  ) \cap (B _ r (q_{2,n})  \setminus B _ r (q_{1,n}) |\, . $$
In turn, the right hand side of the above inequality does not exceed the measure of the region $D _n$ obtained as the difference 
between two right cylinders having both
as axis the straight line containing $q_{1,n}$ and $q _{2,n}$,  as height    $\delta + s _n$, and 
as bases the $(d-1)$-dimensional balls obtained as intersecting $H_{T- \delta}$ respectively with $B _ r (q_{2,n}) $ and 
$B _ r (q_{1,n}) $. 
The measure of such region $D _n$ satisfies ({\it cf.}\ the proof of Lemma \ref{l:nopairs})
$$
 |D_n | \sim  2( d-1)   \omega _{d-1}  r ^ {d-3} 
( t_{2,n} \gamma_n +  \gamma _n ^ 2  ) \big (  t_{2,n} + \gamma _n  \big )
\,, 
$$ where
$\gamma _n$ is the distance of  $q_{1,n}$ and  $q_{2,n}$ from $H _ { T+ s_n}$ (or equivalently, $2 \gamma_n$ is the distance between $q_{1,n}$ and  $q_{2,n}$), and 
$t _ { 2, n}= \delta + s _n - \gamma _n$ is the distance of $q_{2,n}$ from $H _ {T - \delta}$. We infer that
\begin{equation}\label{f:eff1} J_n  
 \leq 8( d-1)   \omega _{d-1}   r ^ {d-3}  \delta ^ 2  \gamma_n  \,,
\end{equation}
where the last inequality holds because $s _n \leq \delta$ for $n$ large enough.

\smallskip
We now turn to estimate $I _n$.   

 Let us begin by proving that:  
\begin{eqnarray} 
 \exists \, \delta _0 >0 \ :\  \inf _n | \Omega \cap H _{T+ \delta_0 + 2 s_n} ^+   \cap (B _ r (q_{1,n})  \setminus B _ r (q_{2 ,n})) |   >0 \,, \  \hskip .1 cm  &\label{f:sguscia}  \\ \noalign{\medskip}
 \exists\,    \eta  >0 \ :\   \inf _n {\rm dist}\,  (U ^ {s_n} _ {T -   \delta_0   }\, ,  \ \overline {\partial ^ * \Omega }  \cap (B _ r (q_{1,n})  \setminus B _ r (q_{2 ,n}))  \geq    \eta    \,,& \label{f:distanziamento}
 \end{eqnarray}  
where stands for the distance in the halfspace $H _{T+ \delta_0 + 2 s_n} ^+$.   

If  the infimum in \eqref{f:sguscia} is zero,  by virtue of \eqref{f:contenimento} we obtain 
$$ 
\inf _n \frac
{\big |\Omega \cap (H _ { T + {s_n}}  \oplus B _{\delta _0 + s_n } (0) ) \cap ( B _ r (q_{1,n}) \Delta B _ r (q_{2,n})) 
\big | }{\| q_{1,n} -q_{2,n} \| } = \inf _n \frac{\big |\Omega \cap ( B _ r (q_{1,n}) \Delta B _ r (q_{2,n})) \big |  }{\| q_{1,n} -q_{2,n} \| }   \, . 
$$  
This is not possible provided  $\delta _0$ is small enough, because 
the left hand side of the above equality is infinitesimal as $\delta_0 \to 0^+$,  while the right hand side is 
controlled from below by a positive constant thanks to the   nondegeneracy assumption.

If \eqref{f:distanziamento} was false  we could find 
a sequence  $\{ x_{n} \} \subset  \overline {\partial ^ * \Omega }  \cap (B _ r (q_{1,n})  \setminus B _ r (q_{2 ,n}))$ such that
$\lim _n {\rm dist}\,  (U ^ {s_n} _ {T - \delta_0}\, ,  x_n) = 0$.  This is not possible because, up to a subsequence, the limit of  $\{x_{n}\}$ would provide  an away contact point at $T$, against our assumption.  
 
Now,  thanks to  \eqref{f:distanziamento}, we   can  consider a geodesic curve minimizing the distance 
between $U ^ {s_n} _ {T - \delta_0}$ and $\overline {\partial ^ * \Omega }  \cap (B _ r (q_{1,n})  \setminus B _ r (q_{2 ,n})$ 
inside $ \big (B _ r (q_{1,n})  \setminus B _ r (q_{2 ,n})\big )\cap H _{T+ \delta_0 + 2 s_n} ^+.   $ We take the mid-point, say  $y_n$, and we consider the ball  
 $B_{ \frac \eta 4  }(y_n)$.   The set
 $ B_{\frac \eta 4}(y_n)   \cap (B _ r (q_{1,n})  \setminus B _ r (q_{2 ,n}))$
 does not intersect $U ^ {s_n} _ {T - \delta_0}$ and,   by \eqref{f:contenimento} and \eqref{f:sguscia}, it  is contained into $\Omega \cap (B _ r (q_{1,n})  \setminus B _ r (q_{2 ,n}))$.   After noticing that this construction continues to work for all $\delta <\delta _0$,   we infer that
there exists a positive constant $K$ such that 
\begin{equation}\label{f:eff2}
I _n \geq  K  \gamma _n\,.
\end{equation} 

By \eqref{f:decomponi}, \eqref{f:eff1}, and \eqref{f:eff2}  we conclude that that, up to taking  $\delta_0$ smaller,  
it holds 
$$\liminf _{n \to + \infty} \frac{|\Omega \cap B _ r (q_{1,n}) | - |\Omega \cap B _ r (q_{2,n}) | }
{ \| q _{ 1 ,n}- q _{ 2 ,n} \|}   >\frac{3K }{4}  \,.
$$
Thus \eqref{f:liminf2} holds true and the proof of Step 2 is achieved.

 \bigskip
 
\subsection{Proof of Step 3.}  By Step 2, we know that at $t = T$ symmetric inclusion occurs with away contact.  
The proof of Step 3 is obtained by showing the following claims:

\smallskip
$\bullet$ {\it Claim 3a.} 
{\it  If $p'$ is an away contact point and 
 $p$ is  its symmetric about about $H _ T$, 
\begin{eqnarray}&  \big | \big [ B _ r (p') \setminus B _ r (p)  \big ] \cap \big [ \Omega \setminus  \reflexT \big ] \big | = 0 \, , \text{ and hence } |(B_r(p') \setminus B_r(p) ) \cap  \Om ^{ns}| =0  ; & \label{f:diff-balls} 
 \\ \noalign{\smallskip} 
 &\exists \e >0 \ :\ \big | B _ \e (p' ) \cap (\Omega \setminus \mathcal R _ T) \big |  = 0 \,,  \text{ and hence $\Omega ^ s$ is open.} & \label{f:claim}
 \end{eqnarray} 
 }
%
$\bullet$ {\it Claim 3b. Properties \eqref{f:cc} and \eqref{f:studiobordi}} hold.

\bigskip 
{\it Proof of Claim 3a.} We have:
$$\begin{array}{ll} 
& \big | \big [ B _ r (p') \setminus B _ r (p)  \big ] \cap \big ( \Omega \setminus  \reflexT  \big ) \big |  
\\
\noalign{\bigskip} 
\displaystyle 
= & \big |  B _ r (p') \cap \big [ \Omega  \setminus (  \Omega _T  \cup  \reflexT ) \big ] \big | - 
\big |  B _ r (p) \cap \big [ \Omega \setminus (  \Omega _T  \cup \reflexT )  \big ]  \big |
\\ \noalign{\bigskip} 
= & 
 \big |  B _ r (p') \cap  \Omega \big | - \big |  B _ r (p') \cap \big (  \Omega _T \cup  \reflexT \big ) \big |   -
\big |  B _ r (p) \cap \Omega \big | + \big |  B _ r (p) \cap \big ( \Omega _T  \cup \reflexT \big ) \big |  = 0
\,,\end{array}
 $$   
 where the first equality holds since  $ B _ r (p') \setminus B _ r (p) $ does not intersect $\Omega _ T$, while in the last one we have used 
 $r$-criticality and the fact that 
 the two sets  $B _ r (p') \cap ( \Omega _T \cup \reflexT )$ and 
 $B _ r (p) \cap (  \Omega _T  \cup  \reflexT ) $  are reflected of each other about $H_T$. 
We have thus proved \eqref{f:diff-balls}.

\smallskip
In view of \eqref{f:diff-balls}, the equality \eqref{f:claim} is immediate in case $p' \not \in \overline {B _ r }(p)$. 
Therefore, we may prove it having in mind that $p' \in  \overline {B _ r } (p)$. 
\smallskip We claim that
\begin{equation}\label{f:case3} 
0 < \big | \big [ B _ r (p' ) \setminus {B _ r (p)} \big ] \cap \Omega \big | < \big |  B _ r (p' ) \setminus {B _ r (p)} \big | \,.
\end{equation}
Indeed, let us exclude both the equalities
$$ \big | \big [ B _ r (p' ) \setminus  {B _ r (p)} \big ]  \cap \Omega \big  | = 0 \quad \text{ and } \quad  \big | \big [ B _ r (p' ) \setminus {B _ r (p)} \big ] \cap \Omega \big | = \big | B _ r (p' ) \setminus  {B _ r (p)}  \big |\,.
$$
The former cannot hold  since $\Omega$ is not $r$-degenerate.   The latter, in view of  \eqref{f:diff-balls}, would imply 
that   $B _ r (p' ) \setminus  {B _ r (p)} $ is contained into $\reflexT$, and hence $B _ r (p ) \setminus  {B _ r (p')} $ is contained into $\Omega _T$.   
Since  $\Omega _T  \cup \reflexT $  is Steiner-symmetric 
about  $H _T$, this would give (via Fubini Theorem) that $p$ and $p'$  belong to ${\rm int} (\Omega ^ { (1)})$, contradicting the fact that  they belong to  $\overline {\partial ^* \Omega}$. 

\smallskip
As a consequence of \eqref{f:case3}, we observe that
\begin{equation}\label{f:brexit}  \exists y' \in \big [ B_r (p') \setminus \overline{B _ r } (p) \big ] \cap \overline{\partial ^* \Omega} \,. \end{equation}   
Indeed, if  \eqref{f:brexit} was false, $B_r (p') \setminus \overline{B _ r}  (p)$ would be contained either into ${\rm int}  ( \Omega ^ { (1)} ) $ or into ${\rm int}(\Omega ^ { (0)})$, against \eqref{f:case3}. 
Next we observe that, in view of \eqref{f:diff-balls}, the two sets $\Omega$  and $\reflexT$ have the same density at  every point of
$B_r (p') \setminus \overline{B _ r}  (p)$, and hence 
$$\big [ B_r (p' ) \setminus \overline{ B _ r}  (p) \big ]  \cap   \partial ^  * \Omega = \big [ B_r (p' ) \setminus \overline{ B _ r } (p) \big ] \cap \partial ^* \reflexT \,;
$$
consequently, since the set $B_r (p') \setminus \overline{B _ r (p)}$ is open, we have 
\begin{equation}\label{f:same-closure-bdry} 
 \big [ B_r (p') \setminus \overline{B _ r}  (p) \big ] \cap \overline {\partial ^* \Omega}  = \big [ B_r (p') \setminus \overline{B _ r} (p) \big ] \cap \overline {\partial ^* \reflexT}  \,.
\end{equation}

By \eqref{f:brexit} and \eqref{f:same-closure-bdry}, it turns out that $y'$ is itself an away contact point. 
Therefore, denoting by $y$ its symmetric about $H _ T$, 
in the same way as we obtained \eqref{f:diff-balls}, replacing the pair $p, p'$ by the pair $y, y'$, we get
 \begin{equation}\label{f:diff-balls2} 
\big | \big [ B _ r (y') \setminus B _ r (y)  \big ] \cap \big ( \Omega \setminus \reflexT \big ) \big | = 0\,. 
\end{equation} 
Moreover, 
since the set $ B_r (p') \setminus \overline{B _ r } (p)$ is open, for $\e>0$ sufficiently small the ball  $B _ \e ( y')$ is contained into $B_r (p') \setminus \overline{B _ r } (p)$, and hence
\begin{equation}\label{f:specchio}
\exists \e >0 \ :\ B _\e (p') \subset  \big [ B_r (y') \setminus \overline{B _ r} (y) \big ]\,.
\end{equation} 
 By \eqref{f:diff-balls2} and \eqref{f:specchio}, \eqref{f:claim} is proved.  
%

\bigskip {\it Proof of Claim 3b.}  In order to prove \eqref{f:cc}-\eqref{f:studiobordi}, we consider the subsets of $H _ T$ defined by  
$$ \begin{array}{ll}
&
C_T: =\Big \{ m_{ ( p, p') } \ :\  p' \text{ is an away contact point, \ $p$ is its symmetric about $H _ T$} \Big \} \, , 
\\  \noalign{\medskip}
& A _ T :=  \Big \{z\in H_T \ : \ g(z) <0 \Big \}\,, 
\end{array}
$$
where $m _{( p, p')} \in H_T$ denotes the mid-point of the segment $(p,p')$, and
the function $g$ is defined as in Proposition \ref{l:reflection} (iv) (applied with $\omega : = \Omega _ T$ and $H := H _ T$).  

By Proposition \ref{l:reflection} (iv), we know that $g$ is continuous and hence the set  $A_T$ turns out to be open.  
Since Claim 3a.\ implies that $C_T$ 
is  open in $H_T$,   we infer that $C_T$ is  relatively open in $A_T$. On the other hand, since $\overline{ \partial ^* \Omega}$ is a closed set, it easy to check that $C_T$  is also relatively closed in $A_T$. Hence,  $C_T$ consists in a non-empty union of connected components of $A_T$. 
Accordingly, $\overline{\partial ^ * \Omega^ s} \cap (H _ T ^ -  \setminus H _ T)$ (resp., $\overline{\partial ^ * \Omega^ s} \cap (H _ T ^ +  \setminus H _ T))$ 
 is a union of connected sets, which are the images of the open connected components of $C_T$ through the continuous function $g$ (resp., the reflections of such images about $H _ T$). 
Each of these connected sets corresponds to $\overline{\partial ^ * \Omega_i^ s} \cap (H _ T ^ -  \setminus H _ T) $ 
(resp., $\overline{\partial ^ * \Omega_i^ s} \cap (H _ T ^ +  \setminus H _ T) $ ) 
for some open connected component $\Omega ^ s _ i $ of $\Omega ^ s$.  This proves \eqref{f:cc}. Since by \eqref{f:claim} none of the sets $\overline{\partial ^ * \Omega_i^ s} \cap (H _ T ^ {\pm}  \setminus H _ T)  $ can 
 intersect $\overline {\partial ^ * \Omega ^ { ns}}$,   \eqref{f:studiobordi} follows.

 \subsection{Proof of Step 4.}  
 Relying on decomposition $\Omega = \Omega ^ s \sqcup \Omega ^ { ns}$ made in Step 3, we are going to 
analyze in detail the behaviour of the open connected components $\Omega ^ s   _ i$ of $\Omega ^ s $. 
To that aim, we need to set up some additional definitions and  notation. 
  
  \smallskip 
Given two two different open connected components $\Om_i^{s} , \Om_j ^ {s}$ of   $\Om^s$,  we say that $\Om^s_i$ {\it is in   $r$-contact with $ \Om_j ^ s$} if   there exists an away contact point $p' \in \overline{\partial ^* \Om^s _i} \setminus H_T$  such that, denoting by $p$ its symmetric about $H _ T$, it holds
$$  
\big |\big ( B_r(p)\Delta B_r(p')\big ) \cap \Om^s _{j} \big |>0.
$$ 
It is not difficult to check that,   if $\Om^s_i$ is in   $r$-contact with $ \Om_j ^ s$,  $\Om^s_j$  is in   $r$-contact with $ \Om_i^ s$.  

If $\Omega _ i ^ s$ is not in contact with  any other component of $\Omega ^ s$, we say that $\Omega_ i ^ s$ is {\it $r$-isolated}.   

 Since our strategy will require to let the initial hyperplane vary, we will write
$$\Omega = \Omega ^{\nu, s} \sqcup \Omega ^ {\nu, ns} \,,
$$ where the additional superscript $\nu$ indicates the direction of the parallel movement, namely the normal to the initial hyperplane $H_0$ (and the decomposition is always meant with respect  to the parallel hyperplane $H _ T$ at the stopping time $T$ defined in Step 2). 
  
  \smallskip
  The proof of Step 4 is achieved by showing the following claims: 
   
  \smallskip 
  
  $\bullet$ 
{\it Claim 4a.  Given $\nu \in \mathbb S ^ { d-1}$,  let $\Omega _ \flat$ be a  $r$-isolated open connected component  of $\Omega ^ {\nu, s}$.
Then $\Omega _ \flat$ is a ball of radius at least $r/2$,  and  $\Omega \setminus \Omega _\flat$ is $r$-critical and  not $r$-degenerate, unless it has measure zero. }
      
   \smallskip
$\bullet$ {\it Claim 4b. The following family is empty: } 
  $$\mathcal F: = \bigcup _{\nu \in \mathbb S^ { d-1}}  \Big \{ \text{\it open connected components not $r$-isolated of }  \Omega ^ {\nu, s}  \Big \}\,. $$   

\smallskip
$\bullet$ {\it Claim 4c (conclusion).}  {\it $\Omega$ is equivalent to a finite union of balls of radius $R> r/2$, at mutual distance larger than or equal to $r$. } 
 
 \bigskip

 {\it Proof of claim 4a}.  Given $\nu \in \mathbb S ^ { d-1}$, let $\Omega _ \flat$ be a  $r$-isolated open connected component  of $\Omega ^ {\nu, s}$. 
Assume by a moment to know that 
\begin{equation}\label{f:grasse}
\Omega _ \flat \text{  is  $r$-critical and  not $r$-degenerate.}   
\end{equation}
In this case, we can restart our proof, with $\Omega _\flat$ in place of $\Omega$. 
Given an arbitrary direction $\widetilde \nu \in \mathbb S ^ { d-1}$,  we make the decomposition 
$$\Omega _\flat = \Omega _\flat  ^ {\tilde \nu, s} \sqcup \Omega _\flat  ^ {\tilde \nu, ns}\,.$$
  We are going to show that, unless $\Omega _\flat ^ {\tilde \nu, ns}$ is empty,  this decomposition 
splits $\Omega _\flat$ into two open sets, contradicting the connectedness of $\Omega  _\flat$.
Hence   $\Omega _\flat$ is Steiner symmetric about a hyperplane with unit normal $\widetilde \nu$.  
By the arbitrariness of $\widetilde \nu$, we deduce that $\Omega _\flat$ is 
a ball.  (Indeed,  
since $\overline \Omega _\flat$ is  a compact set,  following \cite{V13},  there exists a sequence of Steiner symmetrizations  of it converging to a ball; but since $ \Omega _\flat $  is already Steiner symmetric in every direction, it must coincide with such ball). 
  Since $\Omega _\flat$ is not $r$-degenerate,  the radius of the ball is strictly larger than $r/2$.

Assuming that $\Omega  _\flat ^ {\widetilde \nu, ns}$ is not empty, let us show that 
every point of $\Omega _\flat$ is in the interior of one among the two sets  $\Omega ^ {\widetilde \nu , s} _\flat$ and  $\Omega ^ {\widetilde \nu, ns }  _\flat$.      
  Let us denote by $\widetilde T$ the stopping time defined as in Step 2 for the parallel movement with normal $\widetilde \nu$. Recall from \eqref{f:studiobordi} that 
\begin{equation}\label{f:inclusion-flat}
\overline{\partial ^* \Om_\flat  ^{\widetilde \nu, s}} \cap  \overline{\partial ^* \Om_\flat  ^{\widetilde \nu, ns}} \sq H_{\widetilde T} \,, 
\end{equation}

Let us consider separately the cases when $x \in \Omega _\flat \setminus  H _{\widetilde T}$ and when $x \in \Omega _\flat \cap  H _{\widetilde T}$. 

Let $x \in \Omega _\flat \setminus \widetilde H$. Since $\Omega_\flat $ is open, there exists a ball $B _\e ( x)$ contained into $\Omega_\flat \setminus \widetilde H$. 
It cannot be $0< |\Omega _\flat ^{\widetilde \nu,  s} \cap B _ \e (x) | < |B _\e (x)|$. Otherwise, by Federer's Theorem, $B _\e ( x)  $ would contain points of $\partial ^* \Omega _\flat ^ {\widetilde \nu,  s} \cap \partial ^ * \Omega _\flat ^ {\widetilde \nu,  ns}$, against \eqref{f:inclusion-flat}. 
We deduce that $B _ \e (x)$ is contained either into $\Omega ^ {\widetilde \nu, s}  _\flat$ or into  $\Omega ^ {\widetilde \nu, ns }  _\flat$, namely 
$x$ is an interior point for one among $\Omega ^{\widetilde \nu,  s } _\flat$ and  $\Omega ^ {\widetilde \nu, ns }  _\flat$.    
         
Let now $x \in \Omega _\flat \cap \widetilde H$, and let $B _\e ( x)$ be a ball contained into $\Omega_\flat$. By the same arguments as above, each of the two sets $B _ \e (x) \cap (H _{\widetilde T} ^- \setminus H_{\widetilde T} )$ 
and $B _ \e (x) \cap (H _{\widetilde T} ^+ \setminus  H _{\widetilde T} )$  
must be entirely contained either into $\Omega ^{\widetilde \nu,  s} _\flat$ or into  $\Omega ^ {\widetilde \nu, ns }  _\flat$.  Recalling that $\Omega _\flat ^ {\widetilde \nu, s}$ is  Steiner symmetric about $ H _{\widetilde T}$, we infer that either both sets are contained into $\Omega ^ {\widetilde\nu, s} _\flat$,  or both sets are contained into $\Omega ^ {\widetilde \nu, ns} _\flat$. Then, also in this case $x$ is an interior point for one among $\Omega ^{\widetilde \nu,  s } _\flat$ and  $\Omega ^ {\widetilde \nu, ns }  _\flat$.

\bigskip

  
  \smallskip
  To conclude the proof of Claim 4a., it remains to show that \eqref{f:grasse} holds true and that the same property is valid for $\Omega \setminus \Omega _\flat$, unless it has measure zero. For the sake of clearness, this will be obtained as the final product of 
  three consecutive lemmas.

    \begin{lemma}\label{bfm03}
Given $\nu \in \mathbb S ^ { d-1}$, let $\Omega _ \flat$ be a  $r$-isolated open connected component  of $\Omega ^ {\nu, s}$. Then  
 $$\inf_{x_1,x_2 \in \partial ^* \Om_\flat} \frac{|\Om^{\nu, s}\cap \big (B_r(x_1) \Delta B_r(x_2)\big )|}{\|x_1-x_2\|} >0.$$
   \end{lemma}
  \proof
   \bigskip
 Assume by contradiction that
\begin{equation}\label{f:hyp12}
\inf_{x_1,x_2 \in \partial ^* \Om_\flat } \frac{|\Om^{\nu, s} \cap \big (B_r(x_1) \Delta B_r(x_2)\big )|}{\|x_1-x_2\|} =0.
\end{equation}
Then there exist sequences of distinct points $\{x_{1,n}\}, \{x_{2,n}\}  \subset  \partial ^* \Om_\flat$, with
  $\|x_{1,n} -x_{2,n} \|\ra 0$,  such that  
 $$ \frac{|\Om^{\nu, s} \cap \big (B_r(x_{1,n}) \Delta B_r(x_{2,n})\big )|}{\|x_{1,n}-x_{2,n}\|} \ra 0  \,.$$  
   Up to subsequences, we may assume that $\|x_{1,n}-x_{2,n}\|$ converges to $0$ decreasingly,  and that $\{x_{1,n}\}$ and  $\{x_{2,n}\}$ converge to some point $\overline x \in \overline {\partial ^* \Om_\flat }$, which may belong or not to $H _ T$, being as usual $T$ the stopping time
    defined as in Step 2 for the parallel movement with normal $ \nu$. 
   Let us examine the two cases separately.

   In case $\overline x \not \in H _ T$, we may assume without loss of generality that  $\{x_{1,n}\}, \{x_{2,n}\} \subset H _ T ^ +\setminus H _ T$. 
   Recall that, by  \eqref{f:cc}, the set $\partial^ * \Omega _\flat  \cap (H _ T ^+ \setminus H _ T)$ is connected. 
   Hence for every $n \geq 1$ we can join $x_{1,n}$ to $x_{1,n+1}$  by a continuous arc $\gamma _{1,n}(s)$ contained into $\partial ^* \Omega _\flat  \cap (H _ T ^+ \setminus H _ T)$. We can repeat the same  procedure for the second sequence, 
   constructing a family of continuous arcs $\gamma _{2,n}(s)$  joining $x_{2,n}$ to $x_{2,n+1}$ for every $n \geq 1$. 
  
   We look at the boundaries of the balls of radius $r$ whose centre moves along $\gamma _{1,n}(s)$ and $\gamma_ {2, n}(s)$. Clearly these balls tends to superpose in the limit as $n \to + \infty$, since $\|x_{1,n}-x_{2,n}\|$ decreases to $0$. 
   Moreover, we know from
  \eqref{f:diff-balls} that, during the continuous movement of their centre along 
  along $\gamma _{1,n}(s)$ and  $\gamma _{2,n}(s)$, 
the boundary of these balls cannot cross points of density $1$ for $\Omega ^ {\nu,  ns}$. 
 This property will give us the required contradiction.  More precisely, we argue as follows.
 Since $\Omega$ is not $r$-degenerate, \eqref{f:hyp12} implies 
$$\inf_{x_1,x_2 \in \partial ^* \Om_\flat } \frac{|\Om^{\nu, ns}\cap \big (B_r(x_1) \Delta B_r(x_2)\big )|}{\|x_1-x_2\|} >0.$$
 In particular, for $n=1$, we have   
${|\Om^{\nu, ns}\cap \big (B_r(x_{1, 1}) \Delta B_r(x_{2, 1})\big )|} >0$. 
   Hence we can pick a point $p\in {\rm int}(B_r(x_{1, 1}) \Delta B_r(x_{2, 1}))$ of density $1$ for $\Omega ^ { \nu , ns} $, and   
   a radius $\e >0$  sufficiently small so that 
  \begin{equation}\label{f:alto} 
  | B_\varepsilon (p) \cap \Om^{\nu, ns} | \geq \frac{1}{2}  | B_\varepsilon (p)| \,.
  \end{equation}
Possibly reducing $\e$ we can also assume that 
 $B_\varepsilon (p)  \sq  \big (B_r(x_{1, 1}) \Delta B_r(x_{2, 1})\big )$. 
   Recalling that the boundaries of the balls of radius $r$ whose centre moves along  the continuous arcs $\gamma _{1,n}(s)$ and  $\gamma _{2,n}(s)$ cannot meet $\Omega ^ {ns}$, we infer that, for $n$ large,    
   $$
 B_\varepsilon (p) \cap \Omega ^ {\nu, ns} \sq  B_r(x_{1,n}) \Delta B_r(x_{2,n}) \,;$$
   hence, still for $n$ sufficiently large, 
$$|B_\varepsilon (p) \cap \Omega ^ {\nu, ns} | \leq | B_r(x_{1,n}) \Delta B_r(x_{2,n}) | < \frac{1}{4}  | B_\varepsilon (p)| \, , $$
against \eqref{f:alto}.

\smallskip In case $\overline x \in H _ T$, we proceed in the same way, except that  we cannot ensure any more that both  sequences $\{x_{1,n}\}$ and $\{x_{2,n}\}$ belong to the same halfspace $H _ T ^+$ or $H _ T ^-$. 
Thus, when we construct the continuous arcs $\gamma _{1,n}  $  and $\gamma _{ 2,n}$, they may belong indistinctly to $\partial ^* \Omega _\flat  \cap ( H ^ - _ T \setminus H_T)$ or to
$\partial ^* \Omega _\flat \cap ( H ^ + _ T \setminus H_T)$, but  this does not affect the validity of the proof since the contradiction follows as soon as $x_{1,n}$ and $x_{2,n}$ are close enough. \qed

\bigskip

    \begin{lemma}\label{bfm04} Given $\nu \in \mathbb S ^ { d-1}$, let $\Omega _ \flat$ be a  $r$-isolated open connected component  of $\Omega ^ {\nu, s}$. 
  There exists a constant $c_\flat >0$ such that
   \begin{equation}\label{f:ci} 
   | \Om ^{\nu, s} \cap B_r(x) | =c_\flat \qquad \forall x \in \overline{\partial^* \Om _\flat}\,.
   \end{equation} 
   
   \smallskip
Moreover, the constant is the same for any other open connected component   of $\Omega ^ {\nu, s}$  such that the closure of its essential boundary intersects 
$ \overline{\partial^* \Om_\flat}$.  
   \end{lemma} 
   
\proof We argue in a similar way as in the proof of the previous lemma. 
Given  $
   x_1, x_2 \in \overline { \partial^* \Om_\flat} \cap (H_T^+\setminus H _ T) $,  
    by  \eqref{f:cc}, they can be joined by a continuous arc $\gamma(s)$ contained into $\partial ^* \Omega _\flat \cap (H _ T ^+ \setminus H _ T)$.      
By  \eqref{f:diff-balls}, the boundary of the ball of radius $r$ centred at any  point along $\gamma (s)$
    cannot cross points of density $1$ for $\Omega ^ {\nu,  ns}$. 
   We deduce that $B _ r (x_1) \Delta B _ r (x_2)$ cannot contain points of density $1$ for $\Omega ^ {\nu,  ns}$.    
      Since $\Omega$ is $r$-critical, it follows that   $| \Om^{\nu, s} \cap B_r(x_1)|=|\Omega^{\nu, s}  \cap B_r(x_2)| $. 
      By the arbitrariness of $x_1$, $x_2$, we infer that   there exists a   constant $c_\flat ^+>0$ such that
   $ | \Om^{\nu, s} \cap B_r(x) | =c_\flat ^+$ for every $x \in \overline { \partial^* \Om_\flat} \cap( H^+_T \sm H_T)$. 
In the same way, we obtain that there exists a   constant $c_\flat ^->0$ such that
   $ | \Om ^{\nu, s} \cap B_r(x) | =c_\flat ^-$ for every $x \in \overline { \partial^* \Om_\flat} \cap( H^-_T  \sm H_T)$. 
 Since the two sets $\overline { \partial^* \Om _\flat } \cap H^{\pm}_T  $ 
  have common points on $H_T$, we conclude that $c_\flat ^ + = c_ \flat ^-$, 
 The same argument proves also the last assertion of the lemma. 
   \qed

\bigskip \begin{lemma}\label{l:remove} 
Given $\nu \in \mathbb S ^ { d-1}$, let $\Omega _ \flat$ be a  $r$-isolated open connected component  of $\Omega ^ {\nu, s}$. 
  Then $\Omega _\flat$  is  and $r$-critical and not $r$-degenerate. The same assertions hold true for its complement $\Omega \setminus 
\Omega _\flat$,  unless it is of measure zero. 
\end{lemma}

\proof   The fact that $\Omega _ \flat$ is not $r$-degenerate    follows from Lemma \ref{bfm03} and the assumption that $\Omega_\flat$ is $r$-isolated. From equality \eqref{f:ci} in Lemma \ref{bfm04},  and the assumption  that $\Omega_\flat$ is $r$-isolated, we infer that 
 there exists a positive constant $c_\flat$ such that 
$   | \Om_\flat \cap B_r(x) | =c_\flat$ for every  $x \in \overline{\partial^* \Om_\flat}$, namely $\Omega _\flat$ is $r$-critical.  
%
Let us now consider the complement  $\Omega \setminus 
\Omega _\flat$.  Assume it is of positive measure, and hence that $\partial ^* ( \Omega \setminus 
\Omega _\flat )$ is not empty.  
The fact that  $\Omega \setminus \Omega _\flat$ is not $r$-degenerate follows from  the assumption that $\Omega$ itself it is not, combined with the fact that points of $\Omega _ \flat$ and  of $\Omega \setminus \Omega _\flat $  cannot lie at distance smaller than $r$ (again by \eqref{f:diff-balls} and the assumption  that $\Omega_\flat$ is $r$-isolated). Finally, 
it holds  
$| (\Om \setminus  \Om_\flat) \cap B_r(x) | =c  $ for every $x \in \overline{\partial^* (\Om \setminus  \Om_\flat)}$, namely $\Om \setminus \Om _\flat$ is $r$-critical. \qed 
\bigskip

\medskip

\medskip
{\it Proof of Claim 4b.} As a preliminary remark, we observe that the family $\mathcal F$ is at most countable. 
This is an immediate consequence  of the fact that any open set of $\R ^d$ has at most countable connected components, 
and of the fact that, for two different directions $\nu_1$ and $\nu_ 2$, 
it is not possible that a connected component of  $\Om^{\nu _1, s}$ intersects a connected component of  $\Om ^ { \nu _2, s} $ without being equal.

\smallskip
We now prove Claim 4b. by contradiction.

First of all let us show that,  if the family $\mathcal F$ is not empty, it contains an element $\Omega  _ \sharp $ which is Steiner symmetric about $d$ hyperplanes 
with linearly independent normals $\nu _1, \dots, \nu _ d$.

Indeed, let $\mathcal S_k$ denote the family of linear subspaces of dimension $k $ in $\R^d$. For every $k = d-2, d-3, \dots 1$, we are going to associate with a given subspace
$V \in  \mathcal S _k$  an element of $\mathcal F$, which will be denoted by $\Omega ^ k _ V$. These mappings
\begin{equation}\label{f:asso}
\mathcal S _k  \ni V\  \longrightarrow \ \Omega ^ k _V \in \mathcal F 
\end{equation}
are constructed as follows. 

For $k = d-2$, given $V \in \mathcal S_{d-2}$,  we consider all the subspaces $ \widetilde V \in \mathcal S_{d-1}$ containing $V$. 
For every such $ \widetilde  V$, denoting by $\widetilde \nu$ the normal direction to $\widetilde V$, we perform the decomposition $\Om^ { \widetilde \nu, s }  \cup \Om ^ { \widetilde \nu, ns }  $. Since $\mathcal F$ is at most countable, there exist two distinct subspaces $ \widetilde  V_1$ and $  \widetilde V_2$ in $\mathcal S _ {d-1}$ 
such that the corresponding symmetric parts $\Om^ { \widetilde \nu _1, s }$ and $\Om^ { \widetilde \nu _2, s }$  share some 
open connected component. 
We pick one among such shared connected components and we associate it with $V$, denoting it by $\Om_V^{d-2}$. 
Notice that neither the spaces $ \widetilde  V_1$, $ \widetilde  V_2$ nor the shared connected component are unique, so the  definition is made by choice.

For $k = d -3$, given $V \in \mathcal S_{d-3}$,  we consider all subspaces $  \widetilde  V \in \mathcal S_{d-2}$ containing $V$. 
Since the image of the mapping in \eqref{f:asso}  previously defined for $k = d-2$ is at most countable, there exist  two distinct subspaces
 $  \widetilde  V _1$ and $  \widetilde  V_2$ in $\mathcal S _ {d-2}$ such that $\Omega ^ {d-2} _{  \widetilde  V_1} =  \Omega ^ {d-2} _{ \widetilde  V _2}$.  We set (again by choice)
  $$\Omega ^ {d-3} _ V:= \Omega ^ {d-2} _{\widetilde V_1} =  \Omega ^ {d-2} _{\widetilde V_2} \,.$$ 
 
We continue the process until we define the  map in \eqref{f:asso} for $k = 1$. 
Arguing as above,  we find two distinct
$\widetilde V_1$ and $\widetilde V_2$ in $\mathcal S _ 1$ such that $\Om_{\widetilde V_1}^1=\Om_{\widetilde V_2}^1$. 
We set 
  $$\Omega _\sharp:= \Omega ^ {1} _{\widetilde V_1} =  \Omega ^ {1} _{\widetilde V_2} \,.$$ 
By construction $\Om_\sharp$ is Steiner symmetric with respect to $d$  hyperplanes with independent normals $\nu_1, \dots, \nu _ d$.  

\smallskip Next we consider 
any other element $\Omega _ {\sharp \sharp}$  of $\mathcal F$ which is in $r$-contact with $\Omega _\sharp$ in the decomposition with respect to one among the directions $\nu_1 , \dots, \nu _d$,  say $\nu _ 1$. 
If $T_1$ is the stopping time for the parallel movement with normal $\nu _ 1$, 
there exist $p, p' \in \overline{\partial ^* \Om_\sharp} \setminus H_{T_1} $, symmetric about $H _{T_1}$,  such that
$$
\big |\big ( B_r(p)\Delta B_r(p')\big ) \cap  \Om_{\sharp\sharp}  \big |>0.
$$ 
Since we are assuming that $\Om_{\sharp\sharp} $ is   Steiner symmetric with respect to $H_{T_1}$, the above inequality implies that 
$\partial B _ r (p)$ contains points of density $1$ for $\Om_{\sharp\sharp} $. 
In particular, this implies that 
$\Omega _ {\sharp \sharp}$  is itself Steiner symmetric about the same hyperplanes as $\Omega _\sharp$ is.

\smallskip
Then, Lemma \ref{bfm05} below  implies that the set $\Om_\sharp \cup \Om_{\sharp\sharp}$ is connected, yielding a contradiction.

\begin{lemma}\label{bfm05}    
Assume that $\omega \sq \R^d$ is  a bounded open set, Steiner symmetric about $d$ hyperplanes  whose normals  are linearly independent. Then $\omega$ is a connected set containing its centre of mass.
   \end{lemma}
   
    \begin{proof}
   
Let us denote the hyperplanes    by $H _ 1, \dots, H _ d$, and by $ \Pi_{H_k} (x)$  the orthogonal projection from $\R ^d$ onto $H_k$ , for $k = 1, \dots, d$. Starting from a fixed point $x_0 \in \omega$, let us consider 
   the sequence of points defined by  
  $
   x_n := \Pi_{H_k} ( x_{n-1})$  if  $n = k \ [\mbox{mod }d]$. 
    It is easy to check that $\{x _n\}$  converges to the centre of mass $G$ of $\omega$.  In fact, 
let us assume  without losing generality that $G$ is at the origin.
     If   $\alpha_n\in (0, \frac \pi2]$ is the angle between  the normals to  the hyperplanes $H_n$ and $H_{n+1} $  (obtained by cyclically repeating $H _ 1, \dots, H _ d$),  and $d_n$ is the distance of $x_n$ to $H_n \cap H_{n+1}$, we have that
   $\|x_{n+1} \|^2 = \|x_n\|^2-d_n^2\,  \sin^2(\alpha_n)$. 
   Then $ \|x_n\|^2$ is decreasing and  $\{d_n\}$ converges to $0$, since $\{\alpha_n\}$ is a periodic sequence of strictly positive numbers. This readily implies  that ${\rm dist} (x_n, H_k)\to 0$, for every $k=1, \dots, d$, and hence $x_n  \ra 0$.
   
Next we observe that, since $\omega$ is open, there exists $\vps>0$ such that $B_\vps(x_0) \subseteq \omega$. 
   From the assumption that $\omega$ is Steiner symmetric about $H _1, \dots H _ d$, we get that $B_\vps(x_n) \sq \omega$ for every $n$ and, more in general, that 
  \begin{equation}\label{f:tubi} \bigcup_{n\ge 1} ([x_{n-1},x_{n}] \oplus B_\vps(0)) \subseteq \omega.
  \end{equation}

  By \eqref{f:tubi}, it turns out that $B _ \e ( G)$ is contained into $\omega$. Moreover,  $\omega$ is connected because the initial point $x_0$ was arbitrarily chosen, and by \eqref{f:tubi} it can be joined to $G$ by a continuous path contained into $\omega$.  
         \end{proof}

\medskip
{\it Proof of Claim 4c.} 
 We start the procedure by choosing a direction $\nu \in \mathbb S ^ { d-1}$. 
By Claim 4b., we can pick a $r$-isolated open connected component of $\Omega ^ {\nu, s}$, which  by Claim 4a. turns out to be a ball of radius $R_1> r/2$. We remove this ball from $\Omega$. 
By Claim 4a.,  
we are left with a set $\Omega'$ which is still $r$-critical and  not $r$-degenerate  (unless it has measure zero).  So we can restart the process
with  $\Omega '$ in place of $\Omega$. 
Again, by Claim 4b., we can pick a $r$-isolated open connected component of $(\Omega  ') ^ {\nu, s} $, which  by Claim 4a. turns out to be a ball of radius $R_2> r/2$. We remove this ball from $ \Omega'$.  We observe that, 
 since the two balls of radii $R_1$ and $R_2$ that we have extracted from $\Omega$ are are $r$-isolated and $r$-critical, necessarily $R_ 1 = R _ 2 =: R$,  and the balls lie at distance larger than or equal to $r$ from each other.   
Since $\Omega$ has finite measure, we can repeat this process a finite number of times, until when
we are left with a set of measure zero.   \qed

  \bigskip

\begin{remark}\label{rem:h}   A technical extension of  Theorem \ref{t:serrin3} is expected to hold 
when the kernel  $\chi_{B_r(0)}$  is
replaced by a  radially symmetric, decreasing, non negative  function $h$ satisfying suitable assumptions:
any set $\Omega$ with finite measure satisfying the  criticality and  nondegeneracy    conditions,  meant as
$$\int _{\Omega}  h ( x -y) dy = c \quad \forall x \in \partial ^* \Omega \quad \text{ and } \quad
\inf _{x_1, x_2 \in \partial ^* \Omega} \frac{ \int _\Omega | h  (x_1- y)-h (x_2- y) |\, dy\, ,   
}{\| x_1 -x_2 \| }  >0 \,,$$
will be a finite union of balls or a single ball, depending on the structure of the level sets of $h$. 
\end{remark}

\medskip
\noindent 
{\bf Acknowledgments.} The authors are thankful to Gabriele Bianchi, Andrea Colesanti, Jacques-Olivier Lachaud, Rolando Magnanini, Michele Marini,  Micka\"{e}l Nahon,  Berardo Ruffini and Shigeru Sakaguchi for fruitful conversations and comments. 
 The first author  was supported by ANR SHAPO (ANR-18-CE40-0013).


\end{document}